\documentclass[10pt]{amsart}

\usepackage{amsmath,mathrsfs}
\usepackage{amssymb,amsfonts}
\usepackage{url,relsize,xcolor}
\usepackage{tikz,tikz-cd}
\usepackage{enumitem}
\usepackage{xr}
\usepackage[hidelinks]{hyperref}

\theoremstyle{plain}
\newtheorem{theorem*}{Theorem}
\newtheorem{theorem}{Theorem}
\newtheorem{prp}[theorem]{Proposition}
\newtheorem{lemma}[theorem]{Lemma}
\newtheorem{corollary}[theorem]{Corollary}
\newtheorem{claim}[theorem]{Claim}

\numberwithin{equation}{section}
\newtheorem{theoremi}{Theorem}

\theoremstyle{definition}
\newtheorem{definition}[theorem]{Definition}

\newtheorem{notation}[theorem]{Notation}

\newtheorem{question}[theorem]{Question}

\theoremstyle{remark}
\newtheorem{remark}[theorem]{Remark}

\DeclareMathOperator{\CH}{\mathsf{CH}}
\DeclareMathOperator{\ZFC}{\mathsf{ZFC}}

\DeclareMathOperator{\Fin}{Fin}

\tikzset{
 symbol/.style={
 draw=none,
 every to/.append style={
 edge node={node [sloped, allow upside down, auto=false]{$#1$}}}
 }
}

\numberwithin{theorem}{section}

\newcommand{\bbN}{{\mathbb N}}

\newcommand{\cM}{{\mathcal M}}

\newcommand{\SI}{{\mathscr I}}
\newcommand{\SJ}{{\mathscr J}}

\newcommand{\wep}{\mathsf{wEP}}
\newcommand{\ncwep}{\mathsf{ncwEP}}

\newcounter{my_enumerate_counter}
\newcommand{\pushcounter}{\setcounter{my_enumerate_counter}{\value{enumi}}}
\newcommand{\popcounter}{\setcounter{enumi}{\value{my_enumerate_counter}}}

\newcommand{\norm}[1]{\left\lVert #1 \right\rVert}

\DeclareMathOperator{\OCA}{{\mathsf {OCA}}}

\newcommand{\supp}{\mathrm{supp}}

\DeclareMathOperator{\MA}{{\mathsf {MA}_{\aleph_1}}}

\usetikzlibrary{decorations.pathreplacing}

\newcommand{\cstar}{$\mathrm{C^*}$}

\newcommand{\N}{{\mathbb N}}

\title{The noncommutative weak Extension Principle}

\author{Alessandro Vignati}
\address[AV]{
 Institut de Math\'ematiques de Jussieu - Paris Rive Gauche (IMJ-PRG)\\
 Universit\'e Paris Cit\'e\\
 B\^atiment Sophie Germain\\
 8 Place Aur\'elie Nemours \\ 75013 Paris, France}
\email{vignati@imj-prg.fr}
\urladdr{http://www.automorph.net/avignati}

\author{Deniz Yilmaz}
\address[DY]{
Institut de Recherche en Informatique Fondamentale (IRIF)\\
 Universit\'e Paris Cit\'e\\
 B\^atiment Sophie Germain\\
 8 Place Aur\'elie Nemours \\ 75013 Paris, France}
\email{deniz.yilmaz@irif.fr}
\urladdr{https://denizyilmaz.fr}
\date{\today}

\begin{document}
\begin{abstract}
We introduce and study the noncommutative weak Extension Principle, a lifting principle aiming to characterise $^*$-homomorphisms between coronas of nonunital separable $\mathrm{C}^*$-algebras. While this principle fails if the Continuum Hypothesis is assumed, we show that this principle holds under mild forcing axioms such as the Open Colouring Axiom and Martin's Axiom. Further, we introduce and study the notion of nonmeagre ideals in multipliers and coronas of noncommutative $\mathrm{C}^*$-algebras, generalising the usual notion of nonmeagre ideals in $\mathcal P(\mathbb N)$.
\end{abstract}
\maketitle
\section{Introduction}
Given a \cstar-algebra $A$, its \emph{multiplier algebra} $\cM(A)$ is the unital \cstar-algebra such that when a unital $B$ contains $A$ as an essential ideal, the identity map on $A$ extends uniquely to a $^*$-homomorphism from $B$ to $\cM(A)$ (\cite[II.7.3.1]{Black:Operator}). $\mathcal M(A)$ is in a sense the largest unital \cstar-algebra in which $A$ sits densely. To be precise, if $A$ is nonunital and separable, the multiplier algebra $\mathcal M(A)$ is never separable in norm, yet its unit ball carries a Polish topology, the \emph{strict topology}, in which $A$ is dense. The \emph{corona algebra} $\mathcal Q(A)$ is the quotient $\mathcal M(A)/A$, and we always denote by $\pi_A\colon \mathcal M(A)\to\mathcal Q(A)$ the canonical quotient map. We refer to~\cite[II.7.3]{Black:Operator} and~\cite[\S 13]{Fa:Combinatorial} for a rigorous presentation and a variety of equivalent definitions of $\cM(A)$.

If $X$ is a locally compact topological space and $A = C_0(X)$, then $\cM(C_0(X)) \allowbreak \cong C(\beta X)$ and $\mathcal Q(A)=C(X^*)$, where $\beta X$ is the \v{C}ech--Stone compactification of $X$ and $X^*=\beta X\setminus X$ is its remainder.
Thanks to this correspondence, multipliers and coronas can be viewed as noncommutative analogues of \v{C}ech--Stone compactifications and remainders.

Multipliers and coronas are crucial objects in the modern development of \cstar-algebra theory, as they are capable of coding in a unique way certain asymptotic properties of \cstar-algebras. For example, they are indispensable tools in extension theory and the associated operator theory (after \cite{BrDoFi} and \cite{Arv:Notes}), and they are key in the study of lifting and perturbation properties (e.g., \cite{OlsenPed.Coronas}). Their structure as \cstar-algebras on their own right has been studied from multiple points of view. To mention a few notable ones, the work of Lin, Ng, and others (see for example \cite{lin1991simple}, \cite{lin2004simple}, or \cite{KaftalNgZhang.Minimal}) focused on the ideal structure of multipliers and coronas, and there has been significant work on purely operator algebraic properties such as proper infiniteness and real rank (\cite{KucNgPerera} and \cite{Lin2016corona}), and recently strong self-absorption (\cite{farah2025coronassa}).

Our main focus is on $^*$-homomorphisms between corona algebras of separable nonunital \cstar-algebras. Ideally, to understand a $^*$-homomorphism between $\Phi\colon \mathcal Q(A)\to\mathcal Q(B)$ one desires to find a \emph{well-behaved} lifting, i.e., a map $\tilde\Phi\colon \mathcal M(A)\to\mathcal M(B)$ making the following diagram commute:

\begin{center}
\begin{tikzpicture}[scale=0.7]
\node (A) at (0,0) {$\mathcal M(A)$};
\node (B) at (0,-2) {$\mathcal Q(A)$};
\node (C) at (4,0) {$\mathcal M(B)$};
\node (D) at (4,-2) {$\mathcal Q(B)$.};

\draw (A)edge[->] node [left] {$\pi_A$} (B);
\draw (A)edge[->] node [above] {$\tilde\Phi$} (C);
\draw (C)edge[->] node [right] {$\pi_B$} (D);
\draw (B)edge[->] node [above] {$\Phi$} (D);


\end{tikzpicture}
\end{center}

There are different notions of well-behavedness: one can require $\tilde\Phi$ to preserve some of the algebraic or the topological (in strict topology) properties of the multipliers involved. We investigate if, and when, well-behaved liftings exist.

A full classification of all $^*$-homomorphisms between coronas cannot escape set-theoretic considerations. In fact, results of Rudin (\cite{Ru}) on nontrivial autohomeomorphisms of $\omega^*$ (and dually automorphisms of $\ell_\infty/c_0$) and of Phillips and Weaver (\cite{PhWe:Calkin}) on the existence of outer automorphisms of the Calkin algebra $\mathcal Q(H)$, show that assuming the Continuum Hypothesis $\CH$ it is not possible to classify automorphisms of corona \cstar-algebras in any meaningful way. To add to this, Farah, Hirshberg and first-named author proved in \cite{farah2017calkin} that if one assumes $\CH$ then all \cstar-algebras of density at most $2^{\aleph_0}$ embeds into the Calkin algebra $\mathcal Q(H)$. This is the noncommutative analogue of (the dual of) Parovi\v{c}enko's theorem, asserting that under $\CH$ all abelian \cstar-algebras of density at most $2^{\aleph_0}$ embed into $\ell_\infty/c_0$. Such $^*$-homomorphisms constructed from $\CH$ are often intractable (that is, they are not trivial in any meaningful way). In general, under $\CH$ one can use model-theoretic saturation or diagonalisation techniques to produce intractable isomorphisms of corona \cstar-algebras. For more on this, see \cite[\S6]{CoronaRigiditySurvey}.

Here we focus on the situation assuming Forcing Axioms like the Open Colouring Axiom $\OCA$ and Martin's Axiom at level $\aleph_1$, $\MA$. The combination of these two axioms (both incompatible with $\CH$) gives the perfect context for stating and proving rigidity results for massive quotients arising in algebra, topology, and operator algebras. We refer to \cite{CoronaRigiditySurvey} for a thorough discussion on the applications of $\OCA$ and $\MA$ to the theory of liftings. 

In this article, we state the \emph{noncommutative weak Extension Principle}, denoted $\ncwep$, a lifting principle for $^*$-homomorphisms between coronas of separable nonunital \cstar-algebras asserting that these maps are tractable. In layman terms, the $\ncwep$ asserts that all $^*$-homomorphisms between coronas arise as the direct sum of two parts, one tractable (\emph{trivial}, one may say), and one with a large kernel. 

To define our extension principle, we take inspiration from the commutative setting. Farah in \cite{Fa:AQ} introduced the weak Extension Principle $\wep$ to fully characterise maps between \v{C}ech--Stone remainders of zero-dimensional topological spaces, and proved this principle holds assuming $\OCA$ and $\MA$ (see \cite[\S3]{Fa:AQ} and \cite{FaMcK:Homeomorphisms}). Following the new development of lifting techniques (under $\OCA$ and $\MA$) for $^*$-homomorphisms between corona \cstar-algebras (see \cite{vignati2018rigidity} and \cite{mckenney2018forcing}), the authors in \cite{VY:WEP} stated the $\wep$ for maps between remainders of not necessarily zero-dimensional spaces, and showed its validity under Forcing Axioms. These principles are usually stated in terms of topological spaces, and then brought in algebraic form via either Stone or Gel'fand's duality, when they become statements about homomorphisms of massive quotients of algebraic structures (Boolean or \cstar-algebras), and one can take advantage of the strong lifting theorems holding in presence of $\OCA$ and $\MA$. 

In case of noncommutative \cstar-algebras, we jump right away to the search of reasonable liftings for $^*$-homomorphisms between coronas. For this, we need to identify well-behaved maps. As already hinted above, there are two notions of triviality here: a strong algebraic one, isolated in \cite{vignati2018rigidity}, that asks for a lifting preserving as much algebra as possible
, and a (potentially) weaker, topological triviality, focusing on the strict topology of multiplier algebras. We call topologically trivial homomorphisms between coronas simply \emph{Borel} (see Definition~\ref{defin:TopTrivialNC}). All algebraically trivial $^*$-homomorphisms are Borel, but it is not known whether the converse holds\footnote{If we were brave, we would dare to conjecture it.}. Whether there is a Borel non algebraically trivial $^*$-homomorphism between coronas cannot be changed by reasonable forcings, as shown in \cite{vignati2018rigidity}; this statement is strongly tied to Ulam stability perturbation phenomena. We do not focus on this problem here, and stick to topological triviality from now on, but we refer to the end of \cite[\S5]{vignati2018rigidity}, \cite[\S3]{mckenney2018forcing} or \cite[\S5]{CoronaRigiditySurvey} for more information on this subject.

The following is our noncommutative extension principle.

\begin{definition}\label{defin:ncwepIntro}
Let $A$ and $B$ be separable nonunital \cstar-algebras. Let $\Phi\colon\mathcal Q(A)\to\mathcal Q(B)$ be a $^*$-homomorphism. We say that $\Phi$ satisfies the noncommutative weak Extension Principle, and write $\ncwep(\Phi)$, if the following holds: there exists a projection $p\in\mathcal Q(B)$ such that 
\begin{enumerate}[label=(wEP \roman*)]
\item $p$ commutes with the image of $\Phi$,
\item $\Phi_{1-p}\colon \mathcal Q(A)\to (1-p)\mathcal Q(B)(1-p)$ has nonmeagre kernel, and 
\item $\Phi_p\colon\mathcal Q(A)\to p\mathcal Q(B)p$ is Borel.
\end{enumerate}
The principle $\ncwep(\Phi)$ is described by the following diagram:
\begin{center}
\begin{tikzcd}
 & (1-p)\mathcal Q(B)(1-p) \arrow[dr] & \\
\mathcal Q(A) \arrow[dr, "\Phi_{p}" ] \arrow[ur, "\Phi_{1-p}"] \arrow[rr, "\Phi = \Phi_{1-p} \oplus \Phi_{p}"] & & \mathcal Q(B). \\
 & p\mathcal Q(B)p \arrow[ur] &
\end{tikzcd}
\end{center}

We say that the \emph{noncommutative weak Extension Principle} holds, and write $\ncwep$, if $\ncwep(\Phi)$ holds for every $^*$-homomorphism between pairs of coronas of separable nonunital \cstar-algebras.
\end{definition}

The $\ncwep$ cannot follow from $\ZFC$ alone, as nontrivial automorphisms of $\ell_\infty/c_0$ and outer automorphisms of the Calkin algebra do not satisfy the $\ncwep$. As these might exist (e.g., under $\CH$) for the $\ncwep$ to hold we need additional set theoretic assumptions. Further,  this is the best principle one can hope for, as the $\ZFC$ example of Dow (\cite{dow2014non}) of an everywhere nontrivial copy of $\omega^*$ inside $\omega^*$ shows that one cannot get rid of the nontrivial summand $\Phi_{1-p}$ (see Remark~\ref{remark:wep}).
The following is proved in \S\ref{S.ProofWEP}.

\begin{theoremi}\label{thm:mainIntro}
Assume $\OCA$ and $\MA$. Then the noncommutative weak Extension Principle $\ncwep$ holds.
\end{theoremi}

The proof of Theorem~\ref{thm:mainIntro} relies on powerful lifting theorems proved by McKenney and the first author (\cite{mckenney2018forcing}) and the first author (\cite{vignati2018rigidity}), and follows the strategy employed in the commutative setting (see \cite{VY:WEP}). The main difficulties arise in that the noncommutative setting is (as one can expect) technically more demanding than the topological one, and every step requires involved computations. Given a $^*$-homomorphism between coronas $\Phi\colon\mathcal Q(A)\to \mathcal Q(B)$ we at first isolate the projection $p$ required by the $\ncwep$ and prove its main properties (e.g., that $p$ is a projection, that it commutes with the range of the starting $^*$-homomorphisms, and that $\Phi_{1-p}$ has large kernel). This part of the work requires a lifting theorem proved in \cite{mckenney2018forcing}. We then focus on the \emph{trivial} summand $\Phi_p$, aiming to show it is Borel. To do this, we follow closely the strategy of \cite{vignati2014algebra}, yet recent work of De Bondt and the first author (\cite{TrivIsoMetric}) comes in help, as it allows us to skip certain technical steps.

In \S\ref{S.LargeIdeals} we focus on nonmeagre ideals in multipliers and coronas (see Definition~\ref{def:noncommnonmeagre}), those that can arise as kernels of the nontrivial summand $\Phi_{1-p}$ given by the $\ncwep$. These generalise the usual notion of nonmeagre ideals in $\mathcal P(\mathbb N)$ as well as nowhere density in topology (see Proposition~\ref{prop:nonmeagreid}), and thus are a strengthening of essential ideals. Notably, we prove that such ideals cannot exist in coronas of stable algebras (Proposition~\ref{prop:stablenonmeagre}), and that some of the most studied ideals in coronas (as those constructed by Lin in \cite{lin1991simple}, or those arising from traces) are meagre (if improper). In fact, the question of whether nonmeagre proper ideals in coronas of simple \cstar-algebras might exist remains open (see Question~\ref{ques:largesimple}). As a consequence, we obtain a substantial generalisation of the main result of \cite{vaccaro2019trivial} (Theorem 1.3 in there), which characterised endomorphisms of the Calkin algebra under Forcing Axioms.

\begin{theoremi}
Assume $\OCA$ and $\MA$. Let $A$ and $B$ be separable nonunital \cstar-algebras, and assume that $A$ is stable. Then all $^*$-homomorphisms from $\mathcal Q(A)$ to $\mathcal Q(B)$ are Borel.
\end{theoremi}

Among the consequence of \cite[Theorem 1.3]{vaccaro2019trivial} one has the (under suitable axioms) the class of \cstar-algebras embedding in the Calkin algebra is not closure under tensor product and countable unions (Theorem 1.2 in \cite{vaccaro2019trivial}). We prove the correspondent of the first result in Corollary~\ref{cor:tensors}, while the second one relies on Ulam stability considerations for maps between matrix algebras that are not necessarily true for $^*$-homomorphisms between arbitrary separable \cstar-algebras (once again, see \cite[\S5]{vignati2018rigidity}, \cite[\S3]{mckenney2018forcing} or \cite[\S5]{CoronaRigiditySurvey} for more on this). Even though we have some partial results extending Theorem 1.2(2) \cite{vaccaro2019trivial} to general coronas, we currently do not have neat nor sharp statements.

The article ends with remarks on dimension phenomena (\S\ref{S.Dimension}).

\subsection*{Acknowledgements}
We thank Ilijas Farah for useful comments and conversations.  AV is partially funded by the Institut Universitaire de France (IUF) and by the ANR JCJC (Jeunes Chercheuses et Jeunes Chercheurs) project ROAR.

\section{The noncommutative weak Extension Principle}\label{s.Preliminaries}

We start by recalling the weak Extension Principle in the commutative setting.
\begin{definition}
Let $X$ and $Y$ be locally compact noncompact second countable topological spaces. We say that $X$ and $Y$ satisfy the weak Extension Principle, and write $\wep(X,Y)$, if the following happens:

For every continuous map $F\colon X^*\to Y^*$ there exists a partition into clopen sets $X^* =U_0\cup U_1$ and an open with compact closure $V_X\subseteq X$ such that
\begin{itemize}
 \item $F[U_0]$ is nowhere dense in $Y^*$, and
 \item there is a continuous proper function $G\colon X\setminus V_X\to Y$ such that $\beta G\restriction U_1=F\restriction U_1$.
\end{itemize}

By $\wep$ we denote the statement ``$\wep(X,Y)$ holds whenever $X$ and $Y$ are locally compact noncompact second countable spaces".
\end{definition}

Both in its original version (\cite{Fa:AQ}) and its generalisation outside the zero-dimensional case (\cite{VY:WEP}), the principle was stated for maps between powers of \v{C}ech--Stone remainders, but we decided to stick to simplest case for clarity, see \S\ref{S.Dimension} for more on this.

When trying to extend the $\wep$ to the noncommutative setting, we want to translate topological terms to algebraic ones. We write $\pi_X$ (resp. $\pi_Y$) for the canonical quotient map $C_b(X)\to C(X^*)$ (resp. $C_b(Y)\to C(Y^*)$).

We write $\pi_X$ (resp. $\pi_Y$) for the canonical quotient map $C_b(X)\to C(X^*)$ (resp. $C_b(Y)\to C(Y^*)$).

\begin{lemma}\label{lemma:trivial}
Let $X$ and $Y$ be locally compact noncompact second countable topological spaces, and let $\Phi\colon C(Y^*)\to C(X^*)$ be a unital $^*$-homomorphism with dual $F\colon X^*\to Y^*$. The following are equivalent:
\begin{enumerate}
\item there is an open set with compact closure $V_X\subseteq X$ and a continuous proper function $G\colon X\setminus V_X\to Y$ such that $\beta G\restriction X^*=F$;
\item there are positive contractions $a\in C_b(Y)$ and $b\in C_b(X)$ such that $1-b\in C_0(X)$, $1-a\in C_0(Y)$ and a nondegenerate $^*$-homomorphism $\overline{aC_0(Y)a} \to \overline{bC_0(X)b}$ which extends to a $^*$-homomorphism $\tilde\Phi\colon \overline{aC_b(Y)a} \allowbreak \to \overline{bC_b(X)b}$ such that for all $f\in \overline{aC_b(Y)a}$ we have that $
\Phi(\pi_Y(f))=\pi_X(\tilde\Phi(f))$.
\end{enumerate}
\end{lemma}

Homomorphisms as in Lemma~\ref{lemma:trivial} as known as algebraically trivial. This is the best one can ask for, and we cannot ask for the well-behaved lift $\tilde\Phi$ to be a $^*$-homomorphism, even in simple cases. Let for example $X=\mathbb R$ and $Y=(-\infty,0]\cup [1,\infty)$. $X^*$ and $Y^*$ are homeomorphic via the identity map $F$. On the other hand, there is no $^*$-homomorphism $C_0(Y)\to C_0(X)$ inducing the dual of $F$ as in Lemma~\ref{lemma:trivial}.

\vspace{5pt}

If $X$ is a locally compact noncompact topological space and $p\in C(X^*)$ is a projection, there is an open set $U_p\subseteq X$ such that $U_p^*\subseteq X^*$ is the clopen set of which $p$ is the characteristic function of. In this case, $pC(X^*)p=C(U_p^*)$ and $(X\setminus U_p)^*=X^*\setminus U_p^*=U_{1-p}^*$. If $\Phi\colon C(Y^*)\to C(X^*)$ is a $^*$-homomorphism and $p\in C(X^*)$ is a projection, we write $\Phi_p\colon C(Y^*)\to C(U_p^*)$ for the cut-down of $\Phi$ by $p$, that is, $
\Phi_p=p\Phi p$.

Altogether, we have the following operator algebraic description of the $\wep$.

\begin{prp}\label{prop:algWEP}
Let $X$ and $Y$ be two locally compact noncompact second countable topological spaces. The principle $\wep(X,Y)$ is equivalent to the following statement: 

\noindent For every $*$-homomorphism $\Phi \colon C(Y^*)\to C(X^*)$ there exists a projection $p \in C(X^*)$ with associated clopen set $U_p^*$ such that 
\begin{itemize}
\item $\Phi_{1-p}\colon C(Y^*)\to C(U_{1-p}^*)$ has essential kernel, and
\item $\Phi_p\colon C(Y^*)\to C(U_p^*)$ is trivial.
\end{itemize}.
With $\Phi_p=p\Phi p$, we have the following diagram.
\begin{center}
\begin{tikzpicture}[scale=0.7]
\node (A) at (0,2) {$C(U_{1-p}^*)$};
\node (B) at (4,0) {$C(X^*)$};
\node (C) at (-4,0) {$C(Y^*)$};
\node (D) at (0,-2) {$C(U_p^*)$};
\node (E) at (-2.5,1.5) {$\Phi_{1-p}$};
\node (F) at (-2.5,-1.5) {$\Phi_p$};
\draw (C)edge[->] node [above] {$\Phi=\Phi_{1-p}\oplus\Phi_p$} (B);
\draw (C) edge[->] node [left] {} (A);
\draw (A) edge[->] node [left] {} (B);
\draw (C) edge[->] node [left] {} (D);
\draw (D) edge[->] node [below] {} (B);
\end{tikzpicture}
\end{center}
\end{prp}

We now move to the noncommutative setting, where \v{C}ech--Stone compactification and remainder corresponds to multiplier and corona algebras. While referring to \cite[\S II.7]{Black:Operator} for a detailed treatment of these objects, we record here a few useful facts and definitions.

\begin{itemize}
\item If $A$ is a \cstar-algebra, $\mathcal M(A)$ is its multiplier algebra and $\mathcal Q(A):=\mathcal M(A)/A$ is its corona algebra. In case $A$ is unital then $\mathcal M(A) = A$, therefore in what follows we shall assume that $A$ is nonunital. By 
$\pi_A\colon\mathcal M(A)\to \mathcal Q(A)$ we denote the canonical quotient map. For $a,a'\in\mathcal M(A)$, we write $a=_A a'$ if $a-a'\in A$. 
\item In case $A$ is abelian, meaning that $A=C_0(X)$ for some locally compact noncompact topological space $X$, then $\mathcal M(A)=C(\beta X)$ and $\mathcal Q(A)=C(X^*)$. 
\item In case $A=\bigoplus A_n$ for some sequence of unital \cstar-algebras $(A_n)$, then $\mathcal M(A)=\prod A_n$. The corona algebra $\prod A_n/\bigoplus A_n$ is known as the reduced product of the sequence $(A_n)$.
\item The algebra $\mathcal M(A)$ is never separable (unless $A$ is unital), nevertheless it often carries a separable topology. Let $A$ be a separable and nonunital \cstar-algebra, and let $(e_n)\subseteq A$ be an increasing approximate identity of positive contractions for $A$ (this always exists, see \cite{Pede:Corona}). The strict topology on $\mathcal M(A)$ is the topology induced by the seminorms
\[
\ell_n=\norm{ae_n} \text{ and }r_n=\norm{e_na},
\]
for $n\in\bbN$. This is a separable topology, which turns $\mathcal M(A)$ into a standard Borel space and its unit ball, $\mathcal M(A)_{\leq1}$, into a Polish space, in which the unit ball of $A$ sits densely as a Borel subset.
\end{itemize}

The following is the topological notion of triviality we shall focus on.

\begin{definition}\label{defin:TopTrivialNC}
Let $A$ and $B$ be separable nonunital \cstar-algebras, and let $\Phi\colon \mathcal Q(A)\to\mathcal Q(B)$ be a $^*$-homomorphism. We call $\Phi$ Borel if
 \[
\Gamma_\Phi= \{(a,b)\in \mathcal M(A)_{\leq 1}\times\mathcal M(B)_{\leq 1}\mid \Phi(\pi_A(a))=\pi_B(b)\}
 \]
 is Borel in the product of strict topologies.
\end{definition}

Having obtained our running notion of triviality, we want to focus on the nontrivial part of our homomorphisms.

\begin{definition}\label{def:good}
Let $A$ be a separable nonunital \cstar-algebra. We say that a sequence $(e_n)\subseteq A$ is a \emph{good approximate identity} for $A$ if $(e_n)$ is an approximate identity of positive contractions such that 
\begin{enumerate}[label=(AId \arabic*)]
 \item\label{approxid1} for every $n\in\N$ we have that $e_{n+1}e_n=e_n$ and $\norm{e_{n+1}-e_n}=1$, and
 \item\label{approxid2} for every finite interval $I\subseteq\bbN$ there is a positive contraction $h_I\leq (e_{\max I+1}-e_{\min I-2})$ such that $h_I(e_{\max I}-e_{\min I -1})=(e_{\max I}-e_{\min I-1})$ and with the property that $h_Ih_J=0$ whenever $\max I+1<\min J$.
\end{enumerate}
\end{definition}

Every separable \cstar-algebra has such an approximate identity, which can be easily be obtained by taking an approximate identity satisfying condition~\ref{approxid1} and going to a subsequence (see for example \cite[\S1.4]{Pede:C*}). 
The following strengthens the notion of essential ideal (see Proposition~\ref{prop:nonmeagreid}).

\begin{definition}\label{def:noncommnonmeagre}
Let $A$ be a nonunital separable \cstar-algebra. An ideal $\SI\subseteq\mathcal M(A)$ containing $A$ is called \emph{nonmeagre} if for every good approximate identity $(e_n)\subseteq A$ and every partition of $\mathbb N$ into consecutive finite intervals $\bar I = (I_n)$ there is an infinite $L\subseteq\mathbb N$ such that 
\[
\sum_{n\in L}(e_{\max I_n}-e_{\min I_n-1})\in\SI.
\]
We abuse notation and say that an ideal $\SJ\subseteq\mathcal Q(A)$ is nonmeagre if its lifting $\{a\in\mathcal M(A)\mid \pi_A(a)\in \SJ\}$ is a nonmeagre ideal in $\mathcal M(A)$.
\end{definition}

We record the following fact, whose proof is deferred to \S\ref{S.LargeIdeals} (see Lemma~\ref{lem:approxid}).

\begin{lemma}\label{lem:approxid1}
Let $A$ be a nonunital separable \cstar-algebra, and suppose that $\SI\subseteq\mathcal M(A)$ is an ideal containing $A$. Assume that there is a good approximate identity $(E_n)\subseteq A$ such that for every partition of $\mathbb N$ into consecutive finite intervals $\bar I = (I_n)$ there is an infinite $L\subseteq\mathbb N$ such that 
\[
\sum_{n\in L}(e_{\max I_n}-e_{\min I_n-1})\in\SI.
\] 
Then $\SI$ is nonmeagre.
\end{lemma}

We will return to study nonmeagre ideals and their properties in \S\ref{S.LargeIdeals}. For now, we just use their definition to introduce the noncommutative (topological) analogue of the $\wep$. When choosing the projection $p$ as in Proposition~\ref{prop:algWEP}, we want to make sure that $\Phi_p$, the cut-down of $\Phi$ by $p$, is still a $^*$-homomorphism. This can only happen if $p$ commutes with the range of $\Phi$.

\begin{definition}\label{defin:ncwep}
Let $A$ and $B$ be separable nonunital \cstar-algebras. Let $\Phi\colon\mathcal Q(A)\to\mathcal Q(B)$ be a $^*$-homomorphism. We say that $\Phi$ satisfies the noncommutative weak Extension Principle, and write $\ncwep(\Phi)$ if the following holds: there exists a projection $p\in\mathcal Q(B)$ such that 
\begin{enumerate}[label=(wEP \roman*)]
\item\label{wep01} $p$ commutes with the image of $\Phi$,
\item\label{wep02} $\Phi_{1-p}\colon \mathcal Q(A)\to (1-p)\mathcal Q(B)(1-p)$ has nonmeagre kernel, and 
\item\label{wep03} $\Phi_p\colon\mathcal Q(A)\to p\mathcal Q(B)p$ is Borel.
\end{enumerate}
The principle $\ncwep(\Phi)$ is described by the following diagram:
\begin{center}
\begin{tikzcd}
 & (1-p)\mathcal Q(B)(1-p) \arrow[dr] & \\
\mathcal Q(A) \arrow[dr, "\Phi_{p}" ] \arrow[ur, "\Phi_{1-p}"] \arrow[rr, "\Phi = \Phi_{1-p} \oplus \Phi_{p}"] & & \mathcal Q(B). \\
 & p\mathcal Q(B)p \arrow[ur] &
\end{tikzcd}
\end{center}

We say that the \emph{noncommutative weak Extension Principle} holds, and write $\ncwep$, if $\ncwep(\Phi)$ holds for every $^*$-homomorphism between pairs of coronas of separable nonunital \cstar-algebras.
\end{definition}

\begin{remark}\label{remark:wep}
\begin{enumerate}
\item All Borel $^*$-homomorphisms between coronas of abelian \cstar-algebras are trivial (see Theorem~2.8 in \cite{vignati2018rigidity}). This, and the fact that nonmeagre ideals in coronas are essential (Proposition~\ref{prop:nonmeagreid}) show that the $\ncwep$ implies the $\wep$. In fact, as there are essential ideals which are not nonmeagre (e.g., Remark~\ref{rem:differencennmess}), this new principle is a strengthening of the $\wep$. 
\item By \cite[Lemma 7.2]{CoFa:Automorphisms} all Borel automorphisms of the Calkin algebra are inner. The existence of outer automorphisms of the Calkin algebra (or of nontrivial automorphisms of the Boolean algebra $\mathcal P(\N)/\Fin$, a result dating back to the '50s, see \cite{Ru}) gives the failure of the $\ncwep$ under $\CH$. We can therefore at best hope for consistency.
\item Dow in \cite{dow2014non} constructed (in $\ZFC$!) an everywhere nontrivial copy of $\omega^*$ inside itself. Dualising his construction, we get a surjective unital $^*$-homomorphism $\ell_\infty/c_0\to\ell_\infty/c_0$ which does not have a Borel nontrivial cut-down. This shows the necessity of the projection $p$ in the statement of the $\ncwep$, and that a stronger extension principle cannot hold. 
\end{enumerate}
\end{remark}

\section{Proving the \texorpdfstring{$\ncwep$}{ncwep}}\label{S.ProofWEP}
In this section we prove Theorem~\ref{thm:mainIntro}, restated for convenience.

\begin{theorem}\label{thm:ncwep}
Assume $\OCA$ and $\MA$. Then the noncommutative weak Extension Principle $\ncwep$ holds.
\end{theorem}

The whole section is dedicated to the proof of Theorem~\ref{thm:ncwep}. We fix some notation. 

\begin{notation}\label{notation1}
We fix two nonunital separable \cstar-algebras $A$ and $B$, together with $(e_n)_{n\in\N}$ and $(e_n^B)_{n\in\N}$, two good  approximate identities (see Definition~\ref{def:good}) for $A$ and $B$ respectively. For $n\in \N$, we let 
\[
q_n=e_{n}-e_{n-1}
\]
and, for $S\subseteq \bbN$,
\[
q_S=\sum_{n\in S}q_n.
\]
This sum converges in strict topology for every $S\subseteq\bbN$ and $q_\bbN=1_{\mathcal M(A)}$. 

We denote $\pi_A\colon\mathcal M(A)\to\mathcal Q(A)$ the canonical quotient map. For $a,a'\in\mathcal M(A)$, we write $a=_Aa'$ for $a-a'\in A$. The quotient map $\pi_B$ and the equivalence relation $=_B$ are defined in the same way.

We also fix a $^*$-homomorphism $\Phi\colon\mathcal Q(A)\to \mathcal Q(B)$, and we let $\tilde\Phi\colon \mathcal M(A)\to\mathcal M(B)$ be a set theoretic lift for $\Phi$, that is, a map making the following diagram commute:
\begin{center}
\begin{tikzpicture}
 \matrix[row sep=1cm,column sep=2cm] {
\node (A1) {$\mathcal M(A)$};
& \node (A2) {$\mathcal M(B)$};
\\
\node (B1) {$\mathcal Q(A)$};
& \node (B2) {$\mathcal Q(B)$};
\\
};
\draw (A1) edge[->] node [above] {$\tilde\Phi$} (A2) ;
\draw (A1) edge[->] node [left] {$\pi_A$} (B1) ;
\draw (A2) edge[->] node [right] {$\pi_B$} (B2) ;
\draw (B1) edge[->] node [above] {$\Phi$} (B2) ;
\end{tikzpicture}.
\end{center}
\end{notation}

The following is derived from \cite{vignati2018rigidity}. A subset $\SI\subseteq\mathcal P(\bbN)$ is \emph{everywhere nonmeagre} if $\SI\cap\mathcal P(S)$ is nonmeagre whenever $S\subseteq\bbN$. A function between \cstar-algebras is a completely positive contraction it is a contraction whose matrix amplifications preserve positivity, and it is order zero if it preserves orthogonality.

\begin{lemma}\label{lem:liftrig}
Assume $\OCA$ and $\MA$. Then there are a partition of $\bbN$ into consecutive finite intervals $(I_n)$, positive contractions $r_n\in B$, and an everywhere nonmeagre ideal $\SI$ on $\mathbb N$ containing all finite sets and such that
\begin{enumerate}
\item\label{cond1} for every $S\subseteq\mathbb N$ the sum $\sum_{n\in S} r_n$ strictly converges in $\mathcal M(B)$, 
\item\label{cond2} for every $S\in\SI$ we have that 
\[
\Phi(\pi_A(\sum_{n\in S}q_{I_n}))=\pi_B(\sum_{n\in S}r_n),
\]
and
\item\label{cond3} if $|n-m|\geq 2$, $r_nr_m=0$.
\end{enumerate}
\end{lemma}
\begin{proof}
For $i\in\{0,1\}$, consider the map $\rho_i\colon \ell_\infty\to\mathcal Q(B)$ induced by extending linearly the function
\[
\chi_S\mapsto\Phi(\pi_A(\sum_{n\in S}q_{2n+i})),
\]
where $\chi_S\in\ell_\infty$ is the characteristic function on $S$. The map $\rho_i$ is a completely positive order zero contraction. By Theorem 2.17 in \cite{vignati2018rigidity} we can find for each $i\in\{0,1\}$ an everywhere nonmeagre ideal $\SI_i\subseteq\mathcal P(\bbN)$ which contains all finite sets and a strictly continuous map
\[
\tilde\rho_i=\sum\rho_{i,n}\colon\ell_\infty\to\mathcal M(B)
\]
such that $\tilde\rho_i$ lifts $\rho_i$ on elements of $\ell_\infty$ whose support is in $\SI_i$. These maps can be constructed in such a way that there are natural numbers $j_n<k_n$ such that the range of $\rho_{i,n}$ is contained in $\overline{(e_{k_n}^B-e_{j_n}^B)B(e_{k_n}^B-e_{j_n}^B)}$, where $\lim j_n=\infty$. Since almost positive elements of \cstar-algebras are close to positive elements\footnote{This has a precise meaning: for every $\varepsilon>0$ there is $\delta>0$ such that for any \cstar-algebra $C$ if $c\in C$ is such that $\norm{c-c^*}<\delta$ and the spectrum of $(c+c^*)/2$ is contained in $[-\delta,\infty)$ then there is a positive $d\in C$ with $\norm{d-c}<\varepsilon$.}, we can assume that each $\rho_{i,n}(\chi_{2n+i})$ is positive.
Let 
\[
\tilde\rho=\tilde\rho_{0}+\tilde\rho_1\colon\ell_\infty\to\mathcal M(B)
\]
be defined extending linearly the map 
\[
\chi_S\mapsto \tilde\rho_0(\chi_{S_0})+\tilde\rho_1(\chi_{S_1})
\]
where $S_i=S\cap \{2n+i\mid n\in\bbN\}$ for $i\in\{0,1\}$. By the discussion in page 1705 of \cite{vignati2018rigidity} (specifically, Lemma 3.2 and crucially its proof), we can partition $\bbN$ into a sequence of consecutive intervals $(I_n)$ for every $n$ there a finite interval $[k_n,j_n]$ such that 
\[
(e_{k_n}^B-e_{j_n}^B)\tilde\rho(\chi_{I_n})(e_{k_n}^B-e_{j_n}^B)=\tilde\rho(\chi_{I_n})
\]
and if $j_n < k_{n+2}$ for every $n$. Note that this implies that $\tilde\rho(\chi_{I_n})\tilde\rho(\chi_{I_m})=0$ whenever $|n-m|\geq 2$.

Since $\SI_0$ and $\SI_1$ are everywhere nonmeagre, the ideal
\[
\SI = \{S\mid S_0\in\SI_0\text{ and }S_1\in \SI_1\}
\]
is everywhere nonmeagre (this can be viewed using the classical characterisation of nonmeagreness by Jalali--Naini and Talagrand, see \cite[\S3.10]{Fa:AQ}).
Note that 
$\tilde\rho$ lifts $\Phi\restriction \{\pi_A(q_S)\mid S\subseteq\bbN\}$ on $\SI$. Setting $r_n=\tilde\rho(\chi_{I_n})$ we have conditions ~\eqref{cond1}--\eqref{cond3}. 
\end{proof}
We now change our approximate identity in light of Lemma~\ref{lem:liftrig}, and let $e'_n=e_{\max I_n}$ for all $n$. For ease of notation, we rename it back as $e_n=e_n'$, and let again
\[
q_n=e_n-e_{n-1}\text{ and }q_I=e_{\max I}-e_{\min I-1}, \text{ for }I\subseteq\bbN.
\]
The elements $q_n$ and $r_n$ are fixed until the end of this section. 

Lemma~\ref{lem:liftrig} gives an everywhere nonmeagre ideal $\SI$ such that for every $S\in\SI$ we have that 
\[
\Phi(\pi_A(\sum_{n\in S}q_{n}))=\pi_B(\sum_{n\in S}r_n).
\]
Let
\begin{equation}\label{theP}\tag{$\ast$}
    r=\sum_nr_n \text{ and }p=\pi_B(r).
\end{equation}
For $S\subseteq\bbN$, we let $r_S=\sum_{n\in S}r_n$ and $p_S=\pi_B(r_S)$.

We claim that $p$ is the projection witnessing that the principle $\ncwep(\Phi)$ holds. The remainder of this section is dedicated to prove conditions \ref{wep01}--\ref{wep03} from Definition~\ref{defin:ncwep}. 
\begin{lemma}\label{lem:pworks1}
For every $\varepsilon>0$ there is $n$ such that for all $m>n$ we have that $\norm{r_{[m-1,m+1]}r_m-r_m}<\varepsilon$.
Consequently
\begin{enumerate}
\item\label{projcond1} for every $S\subseteq\bbN$
\[
\sum_{n\in S} r_{[n-1,n+1]}r_n=_B r_S,\text{ and}
\]
\item \label{projcond2}$p$ is a projection.
\end{enumerate}
\end{lemma}
\begin{proof}
We argue by contradiction, and suppose that there is $\varepsilon>0$ and an infinite increasing sequence $(n_k)$ such that for all $k\in\bbN$ we have that
\[
\norm{r_{[n_k-1,n_k+1]}r_{n_k}-r_{n_k}}>\varepsilon.
\]
We can assume that $n_{k+1}>n_{k}+3$. Let $J_k=[n_k-1,n_{k+1}-1)$. Since the ideal $\SI$ is nonmeagre, we can find an infinite $L$ such that $\bigcup_{k\in L}J_k\in\SI$, and therefore
\[
\tilde\Phi(\sum_{k\in L}q_{[n_k-1,n_k+1]})=_B\sum_{k\in L} r_{[n_k-1,n_k+1]},
\text{ and }
\tilde\Phi(\sum_{k\in L}q_{n_k})=_B\sum_{k\in L}r_{n_k}.
\]
Using that $\sum_{k\in L}q_{[n_k-1,n_k+1]}\sum_{k\in L}q_{n_k}=\sum_{k\in L}q_{n_k}$, we have that 
\[
\sum_{k\in L}r_{[n_k-1,n_k+1]}r_{n_k}=\sum_{k\in L} r_{[n_k-1,n_k+1]}\sum_{k\in L} r_{n_k}=_B\sum_{k\in L} r_{n_k},
\]
where the first equality is given by that $n_k+3<n_{k+1}$ and the fact that if $|n-m|\geq 2$ then $r_nr_m=0$. Bringing everywhere together we get that
\[
0=\norm{\pi_B(\sum_{k\in L}r_{[n_k-1,n_k+1]}r_{n_k}-\sum_{k\in L}r_{n_k})}=\limsup_k\norm{r_{[n_k-1,n_k+1]}r_{n_k}-r_{n_k}}>\varepsilon.
\]
This is a contradiction.

Let us now show \eqref{projcond1} and \eqref{projcond2}. \eqref{projcond1}: Fix $S\subseteq\bbN$ and let $S_0=S\cap \{2n\mid n\in \bbN\}$ and $S_1=S\setminus S_0$. Note that
\[
\sum_{n\in S} r_{[n-1,n+1]}r_n=\sum_{n\in S_0}r_{[n-1,n+1]}r_n+\sum_{n\in S_1} r_{[n-1,n+1]}r_n \text{ and }r_S=r_{S_0}+r_{S_1}.
\]
Since $r_nr_m=0$ whenever $|n-m|\geq 2$, then
\[
\sum_{n\in S_0} r_{[n-1,n+1]}r_n-r_{S_0} = \sum_{n\in S_0}(r_{[n-1,n+1]}r_n-r_n).
\]
Since $\norm{\pi_B(\sum_{n\in S_0}(r_{[n-1,n+1]}r_n-r_n))}=\limsup_{n\in S_0}\norm{r_{[n-1,n+1]}r_n-r_n}=0$, we get that $\sum_{n\in S_0} r_{[n-1,n+1]}r_n=_Br_{S_0}$. The same reasoning gives that $\sum_{n\in S_1} r_{[n-1,n+1]}r_n=_Br_{S_1}$, we have proved \eqref{projcond1}. 

\eqref{projcond2}: Since $r=r_{2\bbN}+r_{2\bbN+1}$ then $r^2=r_{2\bbN}^2+r_{2\bbN}r_{2\bbN+1}+r_{2\bbN+1}r_{2\bbN}+r_{2\bbN+1}^2$. Since $r_nr_m=0$ if $|n-m|\geq 2$, then $
r_{2\bbN}^2=\sum_nr_{2n}^2$, $r_{2\bbN+1}^2=\sum_nr_{2n+1}^2$, $
r_{2\bbN}r_{2\bbN+1}=\sum_n (r_{2n}+r_{2n+2})r_{2n+1}$, and $r_{2\bbN+1}r_{2\bbN}=\sum (r_{2n-1}+r_{2n+1})r_{2n}$.
Putting everything together we get that
\[
r^2=\sum r_{[2n-1,2n+1]}r_{2n}+\sum r_{[2n,2n+2]}r_{2n+1}=_Br_{2\bbN}+r_{2\bbN+1}=r.
\]
This shows that $p$ is a projection.
\end{proof}
\begin{notation}\label{notation2}
A sequence $\bar I=(I_n)$ of consecutive finite nonempty intervals in $\bbN$ is called a \emph{sparse sequence} if $\max I_n+1<\min I_{n+1}$ for all $n\in\bbN$. For $S\subseteq\bbN$ we write $I_S$ for $\bigcup_{n\in S}I_n$.

Let $\bar I=(I_n)$ be a sparse sequence. Define
\[
\mathcal F(\bar I)_n=\{a\in A\mid q_{I_n}a=aq_{I_n}=a\} \text{ and } \mathcal F(\bar I)=\prod \mathcal F(\bar I)_n.
\]
If $a\in\mathcal F(\bar I)$ we say that $a$ is supported on $\bar I$. Note that if $S\subseteq\N$ is disjoint from $\bigcup_nI_n$ and $a$ is supported on $\bar I$, then $q_Sa=0=aq_S$.

When we write `Let $a=\sum a_n\in\mathcal F(\bar I)$' we implicitly mean that $a_n\in\mathcal F(\bar I)_n$. In this case, if $S\subseteq\bbN$, we let $a_S=\sum_{n\in S}a_n=q_{I_S}aq_{I_S}$. Note that 
\[
\mathcal F(\bar I)\cap A=\bigoplus\mathcal F(\bar I)_n,
\]
meaning that 
\[
\pi_A[\mathcal F(\bar I)]=\prod \mathcal F(\bar I)_n/\bigoplus \mathcal F(\bar I)_n.
\]
Define the sets $\mathcal D(\bar I)\subseteq\mathcal M(B)$ by setting
\[
\mathcal D(\bar I)_n=r_nBr_n \text{ and }
\mathcal D(\bar I)=\prod\mathcal D(\bar I)_n
\]
Once again, we have that
\[
\mathcal D(\bar I)\cap B=\bigoplus \mathcal D(\bar I)_n\text{ and }
\pi_B[\mathcal D(\bar I)]=\prod\mathcal D(\bar I)_n/\bigoplus\mathcal D(\bar I)_n.
\]
\end{notation}

\begin{lemma}\label{lemma:commuting}
Let $\bar I$ be a sparse sequence, and let $a=\sum a_n\in\mathcal F(\bar I)$. The following assertions hold:
\begin{enumerate}
\item\label{commuting11} $r\tilde\Phi(a)=_Br_{I_\N}\tilde\Phi(a)$,
\item\label{commuting12} $r\tilde\Phi(a)r=_B\sum_n (r_{I_n}\tilde\Phi(a)r_{I_n})$,
\item\label{commuting13} $[r,\tilde\Phi(a)]\in B$, and
\item\label{commuting14} for every $S\subseteq \N$ we have that $r\tilde\Phi(a_S)=_Br_{I_S}\tilde\Phi(a)$.
\end{enumerate}
\end{lemma}
\begin{proof}
To ease of notation, let $x=\tilde\Phi(a)$. 
We will use repeatedly the following fact.
\begin{claim}\label{claim:thisclaim}
For every $\varepsilon>0$ and $k\in\N$ there is $n_0$ such that $\norm{r_{S}xe_k^B}<\varepsilon$ if $S\subseteq\N\setminus n_0$.
\end{claim} 
\begin{proof}    
The sequence $(\sum_{n\leq \ell}r_{n}x)_\ell$ converges strictly to $r_{\N}x$, and therefore the sequence $(\sum_{n\leq \ell}r_{n}xe_k)_\ell$ converges in norm, and it is in particular Cauchy.
\end{proof}

\noindent\eqref{commuting11}: Suppose that $r_{\mathbb N\setminus I_{\N}}x\notin B$, and let $\varepsilon>0$ such that $\norm{\pi_B(r_{\mathbb N\setminus I_{\N}}x)}>\varepsilon$. By passing to a subsequence we can find disjoint finite intervals $[j_n,k_n]\subseteq\N$ such that
\[
\norm{(e_{k_n}^B-e_{j_n}^B) r_{\mathbb N\setminus I_{\N}}x (e_{k_n}^B-e_{j_n}^B)}>\varepsilon/2.
\]
By enlarging the intervals and eventually going to a subsequence, we can assume that for each $n$ there is a finite $F_n\subseteq\N\setminus I_\N$ such that $(e_{k_n}^B-e_{j_n}^B) r_{\mathbb N\setminus I_{\N}}=r_{F_n}$. Applying Claim~\ref{claim:thisclaim} repeatedly, we can further pass to a subsequence and assume that $\norm{r_{F_n}x (e_{k_n}^B-e_{j_n}^B)- r_{F_n}x}<2^{-n}$, 
so that
\[
\limsup_n\norm{r_{F_n}x}=\norm{\pi_B(r_{\bigcup F_n}x)}>0.
\]
Let now $S$ be infinite and such that $T:=\bigcup_{n\in S} F_n\in \SI$, which exists by nonmeagreness of $\SI$. Since $T$ is disjoint from $\bigcup I_n$, $q_Ta=0$. Putting everything together we get that 
\[
0=\norm{\Phi(q_Ta)}=\norm{\Phi(q_T)\Phi(a)}=\norm{\pi_B(r_Tx)}>0,
\]
a contradiction.

\noindent\eqref{commuting12}: By \eqref{commuting11}, $rxr=_Br_{I_\N}xr_{I_\N}$. Therefore 
\[
rxr-\sum_n r_{I_n}xr_{I_n}=_B\sum_n (\sum_{m\neq n}r_{I_n}xr_{I_m}).
\]
To show condition \eqref{commuting12}, it does suffice to show that for every $\varepsilon>0$ there is $n_0$ such that for every disjoint finite sets $F,G\subseteq\N\setminus n_0$ we have that $\norm{r_{I_F}xr_{I_G}}<\varepsilon$. If this is not the case, we can find two sequences of finite nonempty sets $(F_n)$ and $(G_n)$ such that $F_n\cap G_m=\emptyset$ for all $n,m\in\N$, and $\norm{r_{I_{F_n}}xr_{I_{G_n}}}\geq\varepsilon$ for all $n$. Once again passing to a subsequence, by nonmeagreness, we can assume that $T_1:=\bigcup_n I_{F_n}$ and $T_2:=\bigcup_n I_{G_n}$ are both in $\SI$, and, applying Claim~\ref{claim:thisclaim} repeatedly, that $r_{T_1}xr_{T_2}=_B\sum_n r_{I_{F_n}}xr_{I_{G_n}}$. The sets $T_1$ and $T_2$ are disjoint, and $q_{T_1}$ and $q_{T_2}$ both commute with $a$, and thus $q_{T_1}aq_{T_2}=0$. Once again putting everything together we get that 
\[0=\norm{\Phi(q_{T_1}aq_{T_2})}= \norm{\pi_B(r_{T_1}xr_{T_2})}=\limsup_n \norm{r_{I_{F_n}}xr_{I_{G_n}}}\geq \varepsilon.
\]
This is a contradiction.

\noindent\eqref{commuting13}: Assume $rx-xr\notin B$. By \eqref{commuting11}, $r_{I_\N}x-xr_{I_\N}\notin B$. By the same argument as before, we can find disjoint intervals $[j_n,k_n]$ such that
\[
\norm{(e_{k_n}^B-e_{j_n}^B)(r_{I_\N}x-xr_{I_\N})(e_{k_n}^B-e_{j_n}^B)} >\varepsilon.
\]
Going to a subsequence and eventually enlarging the sets $[j_n,k_n]$ we can assume that there are finite disjoint $F_n\subseteq \N$ such that $(e_{k_n}^B-e_{j_n}^B)r_{I_\N}=r_{I_{F_n}}$ and $\norm{r_{I_{F_n}}x(e_{k_n}^B-e_{j_n}^B)-r_{I_{F_n}}x}<2^{-n}$. Let $S$ be infinite and such that $T:=\bigcup I_{F_n}\in \SI$. Then $q_T$ commutes with $a$, contradicting the fact that $\norm{r_Tx-xr_T}\geq\varepsilon$ and that $r_T$ lifts $\Phi(\pi_A(q_T))$, as $T\in\SI$.

\noindent\eqref{commuting14}: Fix $S$, and enumerate it increasingly as $S=\{n_k\mid k\in\N\}$. Let $J_k=I_{n_k}$. Since $a_S$ is supported on $\bar J$, we have that $r\tilde\Phi(a_S)=_Br_{J_\N}\tilde\Phi(a_S)=r_{I_S}\tilde\Phi(a_S)$. Note that this shows that $r_{\N\setminus S}\tilde\Phi(a_S)\in B$, and therefore the same argument applied to $\N\setminus S$ gives that $r_{S}\tilde\Phi(a_{\N\setminus S})\in B$. On the other hand, $r_{I_S}\tilde\Phi(a)=_Br_{I_S}\tilde\Phi(a_S)+r_{I_S}\tilde\Phi(a_{\N\setminus S})$, and therefore $r_{I_S}\tilde\Phi(a)=_Br_{I_S}\tilde\Phi(a_S)=_Br\tilde\Phi(a_S)$.
\end{proof}

If $\bar I=(I_n)$ is a partition of $\bbN$ into consecutive finite nonempty intervals and $i<4$, we let $I_{n}^i=I_{4n+i}\cup I_{4n+i+1}$. This gives us a sparse sequence $\bar I^i=(I_n^i)$. The following is Lemma~2.6 in \cite{vignati2018rigidity} (this was essentially derived from early work of Elliott's, see e.g. the proof of \cite[Theorem 3.1]{Ell:Derivations} or \cite[Lemma 9.7.6]{Fa:Combinatorial}).
\begin{lemma} \label{lem:stratification}
For every $a\in\mathcal M(A)$ there is a partition of $\bbN$ into consecutive finite nonempty intervals $\bar I$ and, for $i<4$, $a_i\in\mathcal F(\bar I^i)$ such that $a-\sum_{i<4}a_i\in A$.
Moreover if $a$ is positive, so is each $a_i$. \qed
\end{lemma}

\begin{prp}\label{prop:largeker}
The projection $p$ given in \eqref{theP} commutes with the range of $\Phi$, and the kernel of $\Phi_{1-p}$ is a nonmeagre ideal.
\end{prp}
\begin{proof}
Since $\Phi(\pi_A(a))$ commutes with $p$ whenever $a$ is supported on some sparse sequence (Lemma~\ref{lemma:commuting}) and every $a\in\mathcal M(A)$ can be written (modulo $A$) as a sum of 4 elements each supported on some sparse sequence (Lemma~\ref{lem:stratification}), then $p$ commutes with the range of $\Phi$. In particular both $\Phi_p$ and $\Phi_{1-p}$ are $^*$-homomorphisms.

Let us show that $\Phi_{1-p}$ has nonmeagre kernel. By Lemma~\ref{lem:approxid1} it is enough to check nonmeagreness on the good approximate identity $(e_n)$. Let $\bar J=(J_n)$ be a partition of $\N$ into finite consecutive intervals. By Lemma~\ref{lem:liftrig} there is a nonmeagre ideal $\SI$ on $\N$ which contains all finite sets such that if $S\in \SI$ then $\Phi$ and $\Phi_p$ agree on $\sum_{n\in S} q_n$, meaning that $\sum_{n\in S} q_n\in \ker(\Phi_{1-p})$. Since $\SI$ is nonmeagre, there is an infinite $L$ such that $\bigcup_{n\in L} J_n\in \SI$, meaning that $\sum_{n\in L} q_{J_n}\in \ker(\Phi_{1-p})$, as required.
\end{proof}

We are ready to conclude the proof of Theorem~\ref{thm:ncwep}: we have shown conditions~\ref{wep01} and \ref{wep02}, and are left to show that $\Phi_p$ is Borel. The first step is to get rid of the nonmeagre ideal $\SI$. 
We let $\tilde\Phi_p\colon\mathcal M(A)\to r\mathcal M(B)r$ be a lift for $\Phi_p$.

\begin{lemma}
For every $S\subseteq\bbN$ we have that $r_S$ lifts $\Phi_p(\pi_A(q_S))$.
\end{lemma}
\begin{proof}
If $S$ is finite, $r_S\in B$ and $\Phi_p(\pi_A(q_S))=0$, so there is nothing to prove. Fix an infinite $S\{n_k\mid k\in\bbN\}$ where $n_k<n_{k+1}$. By partitioning $S$ by its equivalence classes in the$\mod 3$ relation, we can assume that $n_k+2<n_{k+1}$. Let $I_k=[n_k-1,n_k+1]$, so that $q_S$ is supported on $(I_n)$. 

With $\SI\subseteq\mathcal P(\bbN)$ the nonmeagre ideal from Lemma~\ref{lem:liftrig}, let
\[
\SI'=\{T\subseteq\mathbb N\mid \bigcup_{k\in T}I_k\in\SI\}.
\]
Since $\SI$ is nonmeagre and $(I_n)$ is a sequence of consecutive disjoint finite intervals in $\bbN$, $\SI'$ is a nonmeagre ideal.
Let
\[
\SJ=\{T\subseteq\bbN\mid r_{\bigcup_{k\in T}I_k}(r_S-\tilde\Phi_p(q_S))\in B\}.
\]
Note that for every $T\subseteq\bbN$ we have that
\begin{equation}\label{eqn2}
r_{\bigcup_{k\in T}I_k}r_S=_Br_{\{n_k\mid k\in T\}}\text{ and }r_{\bigcup_{k\in T}I_k}\tilde\Phi_p(q_S)=_{B}\tilde\Phi_p(q_{\{n_k\mid k\in T\}}), 
\end{equation}
where the last equality comes from Lemma~\ref{lemma:commuting}\eqref{commuting14}.

Since $r_S-\tilde\Phi_p(q_S)$ is fixed, the association $T\mapsto r_{\bigcup_{k\in T}I_k}$ is (product-strictly) continuous, and $B\subseteq\mathcal M(B)$ is Borel, then $\SJ$ is Borel. Moreover, if $T\in \SI'$, then $r_{\{n_k\mid k\in T\}}=_B\tilde\Phi_p(q_{\{n_k\mid k\in T\}})$, which implies that $T\in\SJ$. This implies that $\SJ$ is a Borel nonmeagre ideal which includes all finite sets. By \cite[Corollary 3.10.2]{Fa:AQ}, $\SJ=\mathcal P(\bbN)$.
Applying equation~\eqref{eqn2} to $T=\mathbb N$, we have that $r_S=_B\tilde\Phi_p(q_S)$. This is the thesis.
\end{proof}


By Lemma~\ref{lemma:commuting}, if $a\in\mathcal F(\bar I)$, then $\tilde\Phi_p(a)=_B\sum r_{I_n}\tilde\Phi_p(a)r_{I_n}$ and 
\[
\lim_n\norm{r_{I_n}\tilde\Phi_p(a)-r_{I_n}\tilde\Phi_p(a)r_{I_n}}+\norm{r_{I_n}\tilde\Phi_p(a)-\tilde\Phi_p(a)r_{I_n}}\to 0,
\]
meaning that 
\[
\tilde\Phi_p(a)\in\prod \mathcal D(\bar I)_n/\bigoplus\mathcal D(\bar I)_n.
\]
Restricting $\Phi_p$ to $\pi_A[\mathcal F(\bar I)]$ we obtain a function
\[
\prod \mathcal F(\bar I)_n/\bigoplus \mathcal F(\bar I)_n\to \prod \mathcal D(\bar I)_n/\bigoplus\mathcal D(\bar I)_n.
\]
This function has the following property: for every $S\subseteq\bbN$ and $a=\sum a_n$ and $a'=\sum a_n'$ in $ \mathcal F(\bar I)$,
\[
\text{if }\pi_A(a_S)=\pi_A(a_S') \text{ then } p_S\Phi_p(\pi_A(a))=p_S\Phi_p(\pi_A(a')).
\]
In other words, the function is \emph{coordinate respecting} according to Definition~2.1 in \cite{TrivIsoMetric}. The main result of \cite{TrivIsoMetric} asserts that these must necessarily come from sequences of maps $\mathcal F(\bar I)_n\to \mathcal D(\bar I)_n$. The proposition below formalises this discussion; its proof derives from the main result of \cite{TrivIsoMetric}.

\begin{lemma}\label{lem:usingDBV}
Assume $\OCA$ and $\MA$. Let $\bar I$ be a sparse sequence. There are functions $\alpha_{\bar I,n}\colon \mathcal F(\bar I)_n\to \mathcal D(\bar I)_n$ such that
\[
\alpha_{\bar I}:=\sum\alpha_{\bar I,n}\colon\mathcal F(\bar I)\to \mathcal D(\bar I)
\]
lifts $\Phi_p$ on $\mathcal F(\bar I)$.
Moreover, since each $\mathcal F(\bar I)_n$ and each $\mathcal D(\bar I)_n$ is separable and the product topology on $\mathcal F(\bar I)$ coincides with the restriction of the strict topology on $\mathcal M(A)$, by picking a countable $2^{-n}$-dense subset on $\mathcal D(\bar I)_n$ we can assume that each $\alpha_{\bar I,n}$ takes only countably many values and it is (norm-norm) Borel, so that $\alpha_{\bar I}$ is (strict-strict) Borel. \qed
\end{lemma}

The next, and last, step of the proof is to uniformise our well-behaved local liftings. We closely follow the strategy of the end of \S3 in \cite{vignati2018rigidity}. Even better, by making sure to translate notation when appropriate, we can even skip some of the very technical proofs and refer directly to \cite{vignati2018rigidity}. The following is Lemma 3.10 in \cite{vignati2018rigidity}.
\begin{lemma}\label{lemma:agreeing}
Let $\bar I$ and $\bar J$ be sparse sequences, and suppose that $\alpha_{\bar I}=\sum\alpha_{\bar I,n}\colon\mathcal F(\bar I)\to\mathcal D(\bar I)$ and $\alpha_{\bar J}=\sum\alpha_{\bar J,n}\colon\mathcal F(\bar J)\to\mathcal D(\bar J)$ are liftings for $\Phi_p\restriction \pi_A[\mathcal F(\bar I)]$ and $\Phi_p\restriction \pi_A[\mathcal F(\bar J)]$ respectively. Let $\varepsilon>0$. Then there is $n>0$ such that for every contraction $x\in \mathcal F(\bar I)\cap\mathcal F(\bar J)$ with $(\sum_{i\leq n}q_i)x=0$ we have that 
$\norm{\alpha_{\bar I}(x)-\alpha_{\bar J}(x)}<\varepsilon$. \qed
\end{lemma}
By Lemma~\ref{lemma:agreeing}, if $\alpha_{\bar I}$ and $\alpha_{\bar J}$ are liftings for $\Phi_p$ on $\mathcal F(\bar I)$ and $\mathcal F(\bar J)$ respectively, we can modify $\alpha_{\bar J}$ so that it is still a lifting with the property as in Lemma~\ref{lem:usingDBV} and it agrees with $\alpha_{\bar I}$ on the intersection of their domains. 
More notation is needed:
\begin{notation}\label{notation3}
Let $\mathbb P$ be the poset of partitions of $\mathbb N$ into consecutive finite intervals. Recall that if $\bar I\in\mathbb P$ and $i<4$ the sparse sequence $\bar I^i$ is defined by $I^i_n=I_{4n+i}\cup I_{4n+i+1}$.

We order $\mathbb P$ by setting $\bar I\leq\bar J$ if there is $n$ such that for all $m\geq n$ there is $k$ such that $I_n\cup I_{n+1}\subseteq J_k\cup J_{k+1}$. (This order was denoted by $\leq_1$ in \cite{vignati2018rigidity} and by $\leq^*$ in \cite{Fa:Combinatorial}). $(\mathbb P,\leq)$ is a $\sigma$-directed partially ordered set. Moreover, for $\bar I,\bar J\in\mathbb P$, $\bar I\leq \bar J$ implies that
\[
\pi_A[\bigcup_{i<4}\mathcal F(\bar I^i)]\subseteq \pi_A[\bigcup_{i<4}\mathcal F(\bar J^i)].
\]
\end{notation}
Define
$\mathcal X=\{(\bar\alpha,\bar I)\},$
where
\begin{itemize}
\item $\bar I\in\mathbb P$,
\item $\bar\alpha=(\alpha^0,\alpha^1,\alpha^2,\alpha^3)$ where for each $i<4$, $\alpha^i\colon \mathcal F(\bar I^i)\to\mathcal D(\bar I^i)$ is a Borel lifting of $\Phi_p$ on $\mathcal F(\bar I^i)$
\item for every $i\neq j$, $\alpha^i$ and $\alpha^j$ agree on $\mathcal F(\bar I^i)\cap\mathcal F(\bar I^j)$.
\end{itemize}
By Lemma~\ref{lem:usingDBV}, for every $\bar I\in\mathbb P$ there is $\bar\alpha$ such that $(\bar I,\bar\alpha)\in\mathcal X$.

Elements of $\mathbb P$ can be viewed as strictly increasing functions $\mathbb N\to\bbN$. For a strictly increasing $f\in\mathbb N^\mathbb N$ such that $f(0)=0$ we can associate the partition $I_n=[f(n),f(n+1))$. Vice versa, if $\bar I=(I_n)\in\mathbb P$, we let $f\in\bbN^\bbN$ defined by $f(n)=\min I_n$. $\mathbb P$ is thus a subspace of the Polish space $\bbN^\bbN$. Fix now $\bar I\in\mathbb P$. Suppose that $\bar\alpha$ is a quadruple of maps where $\alpha^i\colon\mathcal F(\bar I^i)\to\mathcal D(\bar I^i)$, for $i<4$. As each $\mathcal F(\bar I^i)_n$ is a subset of $A$, it is separable, and we can thus see $\alpha^i$ as an element of the Polish space $\prod\mathcal F(\bar I^i)_n\to\prod \mathcal D(\bar I^i)_n$. This association gives $\mathcal X$ a separable metrizable topology $\tau$.

For $n\in\bbN$ we partition unordered pairs in $\mathcal X$ by setting
\[
[\mathcal X]^2=L_0^n\cup L_1^n
\]
where $\{(\bar I,\bar\alpha),(\bar J,\bar\beta)\}\in L_0^n$ if and only if there are $m\in\bbN$, $i,j<4$ and a contraction $x\in \mathcal F(\bar I^i)\cap \mathcal F(\bar J^j)$ with $(\sum_{k\leq m} q_i)x(\sum_{k\leq m} q_i)=x$ such that
\[
\norm{\alpha^i(x)-\beta^j(x)}>2^{-n}.
\]
Each $L_0^n$ is open when viewed as a subspace of the product $\mathcal X^2$ (when $X$ is given the topology $\tau$ discussed above). 

Comparing $\mathcal X$ and the partitions $[\mathcal X]^2=L_0^n\cup L_1^n$ with the equally named objects defined in Notation 3.9 in \cite{vignati2018rigidity}, we get the following, which is \cite[Lemma 3.11]{vignati2018rigidity}. (The cardinal $\mathfrak b$ is the least cardinality of a family in $\bbN^\bbN$ which is unbounded in the order of almost domination or, equivalently, the least cardinality of a $\leq$-unbounded set in $\mathbb P$.)
\begin{lemma}\label{lem:noL0hom}
If $\mathfrak b>\omega_1$ then there is no uncountable $L_0^n$-homogeneous set.\qed
\end{lemma}
The following encompasses the discussion after Proposition 3.12 in \cite{vignati2018rigidity}.
\begin{lemma}
Assume $\OCA$. We can find sets $D_k\subseteq\mathcal Y_k\subseteq\mathcal X$ such that
\begin{itemize}
 \item $D_k$ is a countable dense subset of $\mathcal Y_k$,
 \item Each $\mathcal Y_k$ is $L_k^1$-homogeneous and $\{\bar I\mid \exists \bar\alpha ((\bar\alpha,\bar I)\in\mathcal Y_k )$ is $\leq$-cofinal in $\mathbb P$.
\end{itemize}
\end{lemma}
\begin{proof}
We apply $\OCA$ to the open partition $L_0^n$, for $n\in\bbN$. First of all, $\OCA$ implies that $\mathfrak b>\omega_1$, and therefore Lemma~\ref{lem:noL0hom} implies that there are no uncountable $L_0^n$-homogeneous sets in $\mathcal X$. Fix $n$. By applying $\OCA$, we can then write $\mathcal X=\bigcup \mathcal X_m$ where each $\mathcal X_m$ is $L_1^n$-homogeneous. Since the order $\leq$ is $\sigma$-directed, a standard argument (e.g. \cite[Lemma 2.2.2 and 2.4.3]{Fa:AQ}) gives the thesis.
\end{proof}
We continue following \cite{vignati2018rigidity}, and diagonalise using elements of $\mathcal Y_k$ while preserving the property of being an almost lift for $\Phi_p$. What follows is \cite[Lemma 3.15]{vignati2018rigidity}.
\begin{lemma}\label{lem:almostthere}
Let $i<4$, $k\in\bbN$, and let $x\in\mathcal M(A)$ be a contraction. Suppose that there is a sequence $\langle (\bar\alpha_n,\bar I_l)\rangle\subseteq \mathcal Y_k$, and an increasing sequence of naturals $N_l>\max (I_l)_{4l+4}$, where $\bar I_l=(I_l)_n$, with the following properties:
\begin{enumerate}
\item\label{Borelc1} $e_{N_l}xe_{N_l}\in \mathcal F(\bar I_l)$ and 
\item\label{Borelc2} if $l<l'$ and $\max (I_l)_n\leq N_l$ then $(I_l)_n=(I_{l'})_n$.
\end{enumerate}
Let $y_n=q_{I^i_n}xq_{I^i_n}$.
Then 
\[
\pushQED{\qed} 
\norm{\pi_B(\sum (\alpha^i_n)_n(y_n))-\Phi_p(\pi_A(x))}\leq 4\cdot2^{-k}.
\qedhere
\popQED
\]
\end{lemma}

We can now conclude our proof by showing that condition~\ref{wep03} from Definition~\ref{defin:ncwep} holds. 
\begin{lemma}\label{lem:Borel}
Let $(x,y)$ be a pair of contractions in $\mathcal M(A)\times p\mathcal M(B)p$. 
The following conditions are equivalent:
\begin{enumerate}
\item\label{Borelc11} $(x,y)\in\Gamma_{\Phi_p}$.
\item\label{Borelc12} For every $k\in\bbN$ there are contractions $x_i\in\mathcal M(A)$ and $y_i\in\mathcal M(B)$, for $i<4$ such that $x=_A\sum_{i<4}x_i$, $y=_B\sum_{i<4}y_i$, and there are sequences $\langle(\bar\alpha_l,\bar I_l)\rangle\subseteq D_k$ and $(N^i_l)\subseteq\bbN$ with $N^{i}_l\geq\max (I_l)_{4l+4}$ and satisfying 
\begin{enumerate}[label=(\alph*)]
\item\label{a1} $e_{N^{i}_l}x_{i}e_{N^{i}_l}\in \prod_n\mathcal F(\bar I_l^i)_n$ 
\item\label{a2} if $l<l'$ and $\max (I_l)_n\leq \max_{i,j}N^i_l$ then $(I_l)_n=(I_{l'})_n$, and
\item\label{a3} 
\[
\norm{\sum (\alpha^i_{l})_l(q_{I_l^i}x_iq_{I_l^i})-y_{i}}<20\cdot 2^{-k}.
\]
\end{enumerate}
\item\label{Borelc13} For all contractions $x_{i}\in\mathcal M(A)$ and $y_{i}\in\mathcal M(B)$, for $i<4$, if $x=_A\sum_{i<4}x_i$ and for every $k\in\bbN$ there are sequences $\langle(\bar\alpha_l,\bar I_l)\rangle\subseteq D_k$ and $(N^{i}_l)$ with $N_l^i\geq \max(I_l)_{4l+4}$ satisfying \ref{a1}, \ref{a2} and \ref{a3}, then $y=_B\sum_{i<4}y_i$.
\end{enumerate}
Consequently, $\Gamma_{\Phi_p}$ is Borel.
\end{lemma}
\begin{proof}
The equivalence of conditions~\eqref{Borelc11}--~\eqref{Borelc13} was proved in \cite[Theorem 3.16]{vignati2018rigidity}. The last statement follows from that \eqref{Borelc12} gives an analytic definition of $\Gamma_\Phi$, while \eqref{Borelc13} provides a co-analytic one.
\end{proof}

\begin{proof}[Proof of Theorem~\ref{thm:ncwep}]
Fix a $^*$-homomorphism between coronas of separable nonunital \cstar-algebras. 
Let $p$ be given in equation \eqref{theP}, where the elements $(r_n)_n$ are given in Lemma~\ref{lem:liftrig}. 

By Lemma~\ref{lem:pworks1} is a projection. By Proposition~\ref{prop:largeker}, $p$ commutes with the range of $\Phi$ and the kernel of $\Phi_{1-p}$ is a nonmeagre ideal, thus conditions~\ref{wep01} and \ref{wep02} hold. Condition \ref{wep03} is implied by Lemma~\ref{lem:Borel}.
\end{proof}

\section{Nonmeagre ideals in coronas}\label{S.LargeIdeals}
We study nonmeagre ideals in multiplier algebras, their properties, and whether these can exist in particular cases. We repeat Definition~\ref{def:noncommnonmeagre} for the reader's convenience.

\begin{definition}\label{def:noncommnonmeagre2}Let $A$ be a nonunital separable \cstar-algebra. An ideal $\SI\subseteq\mathcal M(A)$ containing $A$ is called \emph{nonmeagre} if for every good approximate identity $(e_n)\subseteq A$ and every partition of $\mathbb N$ into consecutive finite intervals $\bar I = (I_n)$ there is an infinite $L\subseteq\mathbb N$ such that 
\[
\sum_{n\in L}(e_{\max I_n}-e_{\min I_n-1})\in\SI.
\]
We abuse notation and say that an ideal $\SJ\subseteq\mathcal Q(A)$ is nonmeagre if its lifting $\{a\in\mathcal M(A)\mid \pi_A(a)\in \SJ\}$ is a nonmeagre ideal in $\mathcal M(A)$.
\end{definition}

The proof of this lemma was promised in \S\ref{s.Preliminaries} (see Lemma~\ref{lem:approxid1}).

\begin{lemma}\label{lem:approxid}
Let $A$ be a nonunital separable \cstar-algebra, and suppose that $\SI\subseteq\mathcal M(A)$ is an ideal containing $A$. Assume that there is a good approximate identity $(e_n)\subseteq A$ such that for every partition of $\mathbb N$ into consecutive finite intervals $\bar I = (I_n)$ there is an infinite $L\subseteq\mathbb N$ such that 
\[
\sum_{n\in L}(e_{\max I_n}-e_{\min I_n-1})\in\SI.
\] 
Then $\SI$ is nonmeagre.
\end{lemma}
\begin{proof}
One can see the hypotheses as `being nonmeagre w.r.t. to the approximate identity $(e_n)$', and we want to show this condition does not depend on the choice of $(e_n)$. 
We let $(f_n)\subseteq A$ be a second good approximate identity for $A$, and set, for $n\in \N$, $g_n=e_n-e_{n-1}$ and $h_n=f_n-f_{n-1}$. If $I\subseteq\N$ is a finite interval let $g_I=e_{\max I}-e_{\min I-1}$ and $h_I=f_{\max I}-f_{\min I-1}$. We also fix a sequence of finite disjoint nonempty intervals $\bar I=(I_n)$. We aim to prove that there is an infinite $L\subseteq\N$ such that $
\sum_{n\in L} h_{I_n}\in \SI$.

We construct two strictly increasing sequences of natural numbers $(m_k)$ and $(n_k)$ such that for all $k$ we have that
\[
\norm{g_{[n_k,n_{k+1}]}h_{I_{m_k}}g_{[n_k,n_{k+1})}-h_{I_{m_k}}}<2^{-k}.
\]
Let $m_0=n_0=0$, and suppose that both $n_k$ and $m_{k-1}$ have been constructed. Let $j_k$ be large enough so that $\norm{f_{j_k}e_{n_k+1}-e_{n_k+1}}<2^{-k-1}$, and let $m_{k}$ be such that $j_k+1<\min I_{m_k}$. Since $h_{I_{m_k}}$ and $f_{j_k}$ are orthogonal, then $\norm{h_{I_{m_{k}}}e_{n_k+1}}<2^{-k-1}$. Let $J$ be an interval such that $\norm{g_Jh_{I_{m_{k}}}-h_{I_{m_{k}}}}<2^{-k-1}$. By the above discussion, we can assume that $\min J> n_k$, and we set $n_{k+1}=\max J+1$. This concludes the construction.
Note that for every infinite $K\subseteq\N$
\[
\sum_{k\in K} g_{[n_k,n_{k+1}]}h_{I_{m_k}}g_{[n_k,n_{k+1}]} =_A \sum_{k\in K}h_{I_{m_k}}.
\]
Let now $J_k=[n_k,n_{k+1})$. Since $\SI$ is nonmeagre (w.r.t. $(e_n)$), we can find an infinite $L$ be such that $\sum_{k\in L}g_{J_k}\in \SI$, and so does $\sum_{k\in L} g_{[n_k,n_{k+1}]}h_{I_{m_k}}g_{[n_k,n_{k+1}]}$. This concludes the proof.
\end{proof}

The following is the noncommutative analogue of the fact that nonmeagre ideals in $\mathcal P(\N)$ containing all finite sets are dense (tall), where an ideal $\SI$ on $\N$ is dense if every infinite subset of $N$ contains an infinite set in $\SI$.

Recall that an ideal $\SI$ in a \cstar-algebra $A$ is essential if its annihilator is trivial, or, equivalently, if $\SI\cap \SJ\neq 
\{0\}$ for every ideal $\SJ\subseteq A$ (see \cite[II.5.4.7]{Black:Operator}).

\begin{prp}\label{prop:nonmeagreid}
All nonmeagre ideals in coronas of separable nonunital \cstar-algebras are essential. 
\end{prp}
\begin{proof}
Let $\SI \subseteq \mathcal{Q}(A)$ be a nonmeagre ideal, and let $\SJ$ be a nonzero ideal in $\mathcal Q(A)$. We want to find a nonzero $a\in \SI\cap \SJ$. Fix a nonzero positive $a\in \SJ$. By Lemma~\ref{lem:stratification} we can find a sparse sequence $\bar I=(I_n)$ and a nonzero positive $b=\sum b_n\in \mathcal F(\bar I)$ such that $\pi_A(b)\leq a$, so that $\pi_A(b)\in \SJ$. We can assume that $1 > \norm{b_n}>\varepsilon$ for some fixed $\varepsilon>0$. Let $J_n=[\min I_n-1, \max I_n+1]$. Since $\SI$ is nonmeagre, we can find an infinite $L$ such that $\pi_A(\sum_{n\in L}(e_{\max J_n}-e_{\min J_n-1}))\in \SI$. Letting $b_L=\sum_{n\in L}b_n$ we have that $\norm{b_L}\geq\varepsilon$. Since $\pi_A(b_L)\leq \pi_A(b)$, $\pi_A(b_L)\in \SJ$, and since $\pi_A(b_L)\leq\pi_A(\sum_{n\in L}(e_{\max J_n}-e_{\min J_n-1}))$, then $\pi_A(b_L)\in \SI$. This concludes the proof.
\end{proof}
\begin{remark}\label{rem:differencennmess}
Even in coronas of abelian \cstar-algebras the two concepts do not coincide. In fact, there are many essential ideals which are not nonmeagre. For example, let $X=[0,\infty)$ and $A=C_0(X)$. The ideal $\SI=\{\pi_A(f)\mid \lim_n f(n)=0\}$ is essential yet meagre.
\end{remark}

It is natural to ask whether (and when) such ideals can exist. Easy examples arise from reduced products.
\begin{lemma}
Let $A_n$ be a sequence of unital \cstar-algebras, and let $A=\bigoplus A_n$, so that $\mathcal M(A)=\prod A_n$. If $\SJ\subseteq\mathcal P(\N)/\Fin$ be a nonmeagre ideal containing all finite sets. Then $\SI=\{a\in\prod A_n\mid \supp(a)\in \SJ\}$ is a nonmeagre ideal in $\mathcal M(A)$. \qed
\end{lemma}

It turns out that reduced products are essentially the only examples in which we can construct nonmeagre ideals. From now on, we focus on showing that in certain classes of coronas these ideals cannot exist.

Fix two positive elements $a$ and $b$ in a \cstar-algebra $A$, and let $\varepsilon>0$. We write
\begin{itemize}
\item $a\preceq b$ if $a$ is Cuntz below $b$, meaning that there is a sequence $x_n$ such that $\norm{x_nbx_n^*-a}\to 0$,
\item $a\lessapprox b$ if there is $x\in A$ such that $xx^*=a$ and $x^*x\in \overline{bAb}$
\item $a\lessapprox_{\varepsilon} b$ if there are $x\in A$ and $z\in\overline{bAb}$ such that $xx^*=a$ and $\norm{x^*x-z}<\varepsilon$.
\end{itemize}

\begin{prp}\label{prop:nomeagreideals}
Let $A$ be a nonunital $\sigma$-unital \cstar-algebra together with a good approximate identity $(e_n)$. Suppose that for every $n,m\in\N$ and $\varepsilon>0$ there is a finite interval $I\subseteq\N\setminus (m+1)$ such that $(e_{n+1}-e_n)\lessapprox_{\varepsilon} e_{\max I}-e_{\min I}$. Then $\mathcal M(A)$ has no proper nonmeagre ideals.
\end{prp}

\begin{proof}
Let $g_n=e_n-e_{n-1}$ and for $S\subseteq\N$ let $g_S=\sum_{n\in S}g_n$. Let $\SI$ be a nonmeagre ideal in $\mathcal M(A)$. The goal is to show that $g_{5\N+j}\in \SI$ for every $j<5$, where $5\N+j=\{5n+j\mid n\in\N\}$. As $1=\sum_{j<5}g_{5\N+j}$, this suffices. We write the proof in case $j=0$. We will show that $g_{5\N}^2\in \SI$; this suffices by functional calculus.

Let $\varepsilon_i=2^{-i}$. We construct a sequence of natural numbers $n_i$ in the following way: $n_0=0$. If $n_i$ has been constructed, we let $n_{i+1}$ be a natural such that there are intervals $J_{i,k}\subseteq [n_i+1,n_{i+1}-2)$, for $k\leq i$, such that 
\[
g_{5k}\lessapprox_{\varepsilon_i}g_{J_{i,k}}.
\]
Write $J_{j,k}^+$ for $[\min J_{i,k}-1,\max J_{i,k}+1]$. 
In particular there are elements $x_{i,k}$ and $y_{i,k}$ in $A$ such that for all $k\leq i$
\[
x_{i,k}x_{i,k}^*=g_{5k}, \, \, \norm{x_{i,k}x_{i,k}^*-y}<\varepsilon_i \text { and }g_{J_{i,k}^+}yg_{J_{i,k}^+}=y.
\]
Since $\SI$ is nonmeagre, we can find an infinite $L$ such that $g_{\bigcup_{i\in L}[n_i,n_{i+1})}\in \SI$. Enumerate $L=\{\ell_i\mid i\in\N\}$. Let 
\[
y_k=g_{[5k-1,5k+1]}x_{\ell_k,k}g_{J_{\ell_k,k}^+}.
\]
Note that $y_ky_{k'}^*=y_{k}^*y_{k'} = 0$ for every $k\neq k'$, hence 
\[
(\sum y_k)(\sum y_k^*)=\sum y_ky_k^*\text{ and }(\sum y_k^*)(\sum y_k)=\sum y_k^*y_k.
\]
On the other hand, $\sum y_k^*y_k\in \prod_k g_{J_{\ell_k,k}^+}Ag_{J_{\ell_k,k}^+}\subseteq\SI$, while $\sum y_ky_k^*-g_{5\N}=\sum (y_ky_k^*-g_{5k})$. Since the elements $y_ky_k^*-g_{5k}$ are mutually orthogonal and $\norm{y_ky_k^*-g_{5k}}\to 0$, we have that $\sum y_ky_k^*-g_{5\N}\in A$. Putting all of these together, we have that, modulo $A$,
\[
g_{5\N}^2=(\sum y_k)(\sum y_k^*)(\sum y_k)(\sum y_k^*)=(\sum y_k)(\sum y_k^*y_k)(\sum y_k^*)\in \SI.
\]
This concludes the proof.
\end{proof}

Recall that a \cstar-algebra is stable if $A\otimes\mathcal K\cong A$, where $\mathcal K$ is the algebra of compact operators on a separable Hilbert space. The following useful characterisation of stability is Theorem~2.1 in \cite{HjeRordam}.
\begin{lemma}\label{lemma:rordamstable}
Let $A$ be a $\sigma$-unital \cstar-algebra. The following are equivalent:
\begin{itemize}
 \item for every $a\in A_+$ such that there is $e\in A_+$ with $ea=a$, there is $x$ such that $xx^*=a$ and $ax^*x=0$;
 \item $A$ is stable.\qed
\end{itemize}
\end{lemma}
\begin{prp}\label{prop:stablenonmeagre}
If $A$ is a stable $\sigma$-unital \cstar-algebra, then $\mathcal M(A)$ has no proper nonmeagre ideal.
\end{prp}
\begin{proof}
We will show that $A$ satisfies the hypothesis of Proposition~\ref{prop:nomeagreideals}. Let $(e_n)$ be an approximate identity for $A$, and fix $n$. We want to show that for every $\varepsilon>0$ and $m\in\mathbb N$ we have that $g_n=e_n-e_{n-1}\lessapprox_{\varepsilon} e_{\max I}-e_{\min I}$ for some finite interval $I\subseteq\N\setminus (m+1)$. Fix $\varepsilon$ and $m$ with $m>n$. By Lemma~\ref{lemma:rordamstable} we can find $x$ such that $xx^*=e_{m+1}$ and $x^*x$ is orthogonal to $e_{m+1}$, meaning that $x^*x\leq 1-e_{m+1}$. Since $x\in A$, we can find a large enough $N$ and $z\in\overline{g_{[m-1,N+1]}Ag_{[m-1,N+1]}}$ such that $\norm{x^*x-z}<\varepsilon$. Let $y=e_n^{1/2}x$, so that $e_n=yy^*$. A simple calculation gives that $\norm{(e_{N+1}-e_{m})y^*-y^*}<3\varepsilon$, thus we can find $z'\in \overline{g_{[m-1,N+1]}Ag_{[m-1,N+1]}}$ with $\norm{y^*y-z'}<5\varepsilon$. As $m$ and $\varepsilon$ are arbitrary, this concludes the proof.
\end{proof}

The next class of interest is that of simple \cstar-algebras. Ideals in multipliers, and consequently coronas, of simple \cstar-algebras were intensively studied (see \cite{lin1988ideals}, \cite{lin1991simple}, \cite{Zhang.Riesz}, and \cite{KucNgPerera}). Notably, Lin isolated in \cite{lin1991simple} a condition named `continuous scale', which detects precisely simplicity of $\mathcal Q(A)$ (see Theorem 2.4 in \cite{lin2004simple}). This and related conditions later found important applications for example in extension theory (\cite{Ng.PICoronasExt}).

Lin also identified a special ideal, denoted $I$ in \cite{lin1991simple} and $I_{\min}$ in \cite{Ng.Minimal}, and defined as the closure of the set
\[
I_0=\{x\in \mathcal M(A)\mid \forall a\in A_+, a\neq 0\exists n_0\forall m>n\geq n_0 \, (g_{[n,m]}xx^*g_{[n,m]}\preceq a)\}
\]
We shall call this ideal $I_{\min}$, to avoid confusion. Several characterisations of $I_{\min}$ were obtained (see \cite{KaftalNgZhang.Minimal} for an overview), and in \cite[Remark 2.9]{lin1991simple} it was shown that $I_{\min}$ is the minimal ideal of $\mathcal M(A)$ containing $A$.

Other important ideals arise from traces. If $\tau$ is a trace on a separable nonunital $A$, then $\tau$ extends to a (not necessarily finite) trace on $\mathcal M(A)$. Define $\SI_\tau$ as the closure of
\[
\SI_{0,\tau}=\{x\in\mathcal M(A)\mid \tau(xx^*)<\infty\}.
\]
$\SI_\tau$ is an ideal in $\mathcal M(A)$ which obviously contains $A$. We want to show these ideals are never meagre (unless they are trivial). First, a lemma.

\begin{lemma}\label{lemma:smallCuntz}
Let $A$ be a \cstar-algebra. Fix $n\in \N$ with $n>0$ and $\varepsilon>0$. Then there is $\delta>0$ with the following property: for all positive contractions $x,a_1,\ldots,a_n,b_1,\ldots,b_n\in A$ such that $b_ib_i^*\preceq a_i$ and $\norm{b_i-x}<\delta$, then there is $b\in A$ with $\norm{b-x}<\varepsilon$ and $bb^*\preceq a_i$ for all $i\leq n$.
\end{lemma}
\begin{proof}
If $n=1$, then $\varepsilon=\delta$ and there is nothing to prove. If $n\neq 1$, let $\delta$ be small enough such that for all positive contractions $a$ and $b$, if $\norm{a-b}<\delta$ then $\norm{a^{1/n}-b^{1/n}}<\varepsilon/n$. Let $b=\prod_{i\leq n}b_i^{1/n}$, so that $\norm{b-x}<\varepsilon$ and $b\preceq b_ib_i^*$ as required.
\end{proof}

\begin{lemma}\label{lem:usualaremeagre}
The ideal $I_{\min}$ and all ideals of the form $\SI_\tau$ are either improper or meagre.
\end{lemma}
\begin{proof}
Let $(e_n)$ be an approximate identity for $A$, and, as before, if $J\subseteq\N$ is a finite interval, write $g_J$ for $e_{\max J}-e_{\min J-1}$.

Say $\SI_\tau$ is proper. Then $\tau(1)=\sup\tau(e_n)=\infty$, and we can thus find disjoint finite intervals $J_n$ such that $\tau(g_{J_n})\geq 1$. Without loss of generality we can assume that $\max J_n<\min J_{n+1}$. Set $K_n=[\min J_n,\min J_{n+1})$. Then there is no infinite $L$ such that $\sum_{n\in L}g_{K_n}\in \SI_\tau$.

Let us now show that if $I_{\min}$ is nonmeagre, then $I_{\min}=\mathcal M(A)$ (in which case, $A$ has a continuous scale and $\mathcal Q(A)$ is simple). We want to show that for every $i<5$ we have that $g_{5\N+i}\in I_{\min}$, and thus $1\in I_{\min}$. Once again, we only check for $i=0$.

For a positive nonzero contraction $a\in A$ and $k\in\N$, define
\[
\varepsilon_{k,a}=\inf\{\norm{b-g_{5k}}\mid b\preceq a\}.
\]
and let $\varepsilon_a=\limsup_k \varepsilon_{k,a}$.
\begin{claim}\label{claimnonmeagre}
$g_{5\N}\in I_{\min}$ if and only if $\varepsilon_a=0$ for every positive nonzero contraction $a\in A$.
\end{claim}
\begin{proof}
If $g_{5\N}\in I_{\min}$, then $g_{5\N}^{1/2}\in I_{\min}$. Fix $\varepsilon>0$, and let $x\in I_0$ with $\norm{x-g_{5\N}^{1/2}}<\varepsilon$. Fix $a\in A_+$ be a nonzero contraction, and let $x_k=g_{[5k-1,5k+1]}xg_{[5k-1,5k+1]}$. Since $x\in I_0$, then for all sufficiently large $k$ we have that $x_kx_k^*\preceq a$. Since $\norm{x_k-g_{5k}^{1/2}}\leq \varepsilon$, then $\norm{x_kx_k^*-g_{5k}}\leq 2\varepsilon^2$. This shows that $\varepsilon_a\leq 2\varepsilon^2$. As $\varepsilon$ and $a$ are arbitrary, $\varepsilon_a=0$ for all relevant $a$.

Vice versa, assume that $\varepsilon_a=0$ for each nonzero positive contraction $a\in A$. Enumerate all positive nonzero contractions as $(a_n)$, for $n\in\N$. Using that $\varepsilon_{a_n}=0$ for all $n$ and applying Lemma~\ref{lemma:smallCuntz} inductively, we can construct a infinite sequence $(n_k)$ and elements $x_i$ with $x_i=g_{[5i-1,5i+1]}x_ig_{[5i-1,5i+1]}$ such that if $i\in [n_k,n_{k+1})$ then $\norm{x_i-g_{5i}}<2^{-k+1}$ and $x_ix_i^*\preceq a_j$ for all $j\leq k$. The element $x=\sum x_i$ is such that $x-g_{5\N}\in A$ and belongs to $I_0$.
\end{proof}
The same argument as in Claim~\ref{claimnonmeagre} shows that for every infinite $S\subseteq\N$ we have that $\sum_{k\in S}g_{5k}\in I_{\min}$ if and only if $\limsup_{k\in S}\varepsilon_{k,a}=0$ for every nonzero positive contraction $a\in A$. If $g_{5\N}\notin I_{\min}$ we can then find $a\in A$, $\varepsilon>0$ and an infinite $S\subseteq \N$ such that $\varepsilon_{k,a}>\varepsilon$ for all $k\in S$. In particular, if $T\subseteq S$ is infinite, then $a$ witnesses that $\sum_{k\in T}g_{5k}\notin I_{\min}$. This contradicts that $I_{\min}$ is nonmeagre.
\end{proof}
\begin{remark}\label{rem:complexity}
An alternative proof of the above proposition goes through the following path. Let $I$ be an ideal in $\mathcal M(A)$ with $A\subseteq I\subseteq\mathcal M(A)$, and fix a sequence $(n_k)\subseteq \N$. Let 
\[
\SJ=\{S\subseteq\N\mid \sum_{k\in S}g_{n_k}\in I\}.
\]
$\SJ$ is an ideal on $\N$ containing all finite sets. If $I$ is nonmeagre according to Definition~\ref{def:noncommnonmeagre} then $\SJ$ is a nonmeagre ideal on $\N$. Since all ideals considered above (Lin's $I$, and all tracial ideals $\SI_\tau$) are strictly Borel and the strict topology when restricted to $\{\sum_{k\in S}g_{n_k}\mid S\subseteq\N\}$ for some increasing sequence $(n_k)$ coincides with the usual product topology on $\mathcal P(\N)$, if the ideals $I$ or $\SI_\tau$ were to be nonmeagre and proper then one could find a sequence $(n_k)$ such that the corresponding ideals on $\N$ would be Borel, proper, and nonmeagre, while containing all finite sets. As these ideals cannot exist (see e.g. \cite[Corollary 3.10.2]{Fa:AQ}), we would get a contradiction. 
\end{remark}

Since all proper ideals that can be constructed `by hand' are strictly Borel and the above argument shows that these ideals cannot be nonmeagre, a positive answer to the following question would rely on the construction of interesting unnatural ideals in multipliers.
\begin{question}\label{ques:largesimple}
Does there exist a simple separable nonunital \cstar-algebra $A$ such that $\mathcal M(A)$ has a proper nonmeagre ideal?
\end{question}

We collect the negative answers to Question~\ref{ques:largesimple} obtained so far.

\begin{prp}\label{prop:nononmeagre}
Let $A$ be a separable nonunital \cstar-algebra. Assume that 
\begin{itemize}
\item $A$ is stable, or
\item $A$ is simple and it has a continuous scale, or
\item $A$ is a simple AF algebra with only finitely many extremal traces.
\end{itemize}
Then $\mathcal M(A)$ does non have improper nonmeagre ideals.
\end{prp}
\begin{proof}
\begin{itemize}
 \item If $A$ is stable, this is a consequence of Proposition~\ref{prop:stablenonmeagre}. 
 \item If $A$ is simple and it has a continuous scale, then $A$ is the only proper ideal of $\mathcal M(A)$, yet clearly $A$ is meagre.
 \item If $A$ is a simple AF algebra with only finitely many extremal traces, then all ideals $A\subseteq I\subseteq\mathcal M(A)$ have the form $I=\SI_{\tau_1}\cap\cdots\cap \SI_{\tau_n}$ for some extremal traces $\tau_1,\ldots,\tau_n$. This follows from Theorem~2 in \cite{lin1988ideals}. Since these cannot be nonmeagre by Lemma~\ref{lem:usualaremeagre}, $\mathcal M(A)$ has no nonmeagre proper ideal.\qedhere
\end{itemize}
\end{proof}

As we have seen, nonmeagre ideals arise as kernels of $^*$-homomorphisms satisfying the noncommutative weak extension principles. When there are no such ideals, we can characterise all endomorphisms between the coronas involved, and thus extend the results Vaccaro obtained in \cite{vaccaro2019trivial} for endomorphisms of the Calkin algebra.
\begin{corollary}
Assume $\OCA$ and $\MA$. Let $A$ and $B$ be nonunital separable \cstar-algebras. Assume that 
\begin{itemize}
\item $A$ is stable, or
\item $A$ is simple and it has a continuous scale, or
\item $A$ is a simple AF algebra with only finitely many extremal traces.
\end{itemize}
Then all $^*$-homomorphisms $\mathcal Q(A)\to\mathcal Q(B)$ are Borel.
\end{corollary}
\begin{proof}
Let $\Phi\colon \mathcal Q(A)\to\mathcal Q(B)$ be a nonzero $^*$-homomorphism. By $\OCA$ and $\MA$, the noncommutative weak Extension Principle $\ncwep$ holds, as witnessed by the projection $p\in\mathcal Q(B)$ and the Borel map $\Phi_p$. Since $\mathcal M(A)$ does not have nonmeagre proper ideals (by Proposition~\ref{prop:nononmeagre}), then $\Phi_{1-p}=0$, and therefore $\Phi=\Phi_p$ is Borel.
\end{proof}

As mentioned in the introduction, the study of ideals in multipliers (and consequently coronas) has been an active topic of research in \cstar-algebras theory for the last three decades, starting from Busby's and Elliott's seminal articles \cite{Busby} and \cite{Ell:Derivations}, and continuing with the work of Lin, Ng, and many others, see e.g. \cite{lin1988ideals}, \cite{lin1991simple}, \cite{lin2004simple}, \cite{Ng.Minimal}, \cite{KucNgPerera} and \cite{ArchboldSomerset}. Other than $I_{\min}$ and tracial ideals, notable ideals arise again from traces (in this case viewed as lower semicontinuous
densely defined tracial weights on $A$) by considering the elements of $\mathcal M(A)$ whose evaluation induces a continuous affine map on $T(\mathcal M(A))$ (for details, see the ideal $I_{cont}$ studied in \cite[\S5]{Ng.Minimal}), or from point evaluations in $C_0(X)$-algebras (see \cite{ArchboldSomerset}). We do not know whether these ideals can be nonmeagre, but we suspect this is not the case, as they seem to have a Borel, or at least an analytic, definition in strict topology, in which case one could follow the argument in Remark~\ref{rem:complexity} to show these cannot be nonmeagre if they are improper. A systematic study of nonmeagre ideals in multipliers and coronas is outside the scope of this article, but will be the topic of future research. 

\section{Noncommutative dimension phenomena}\label{S.Dimension}

The original statement of the weak Extension Principle was made in terms of maps between powers of \v{C}ech--Stone remainders. In the commutative setting, proving instances of such a principle amounts in studying maps $(X^*)^d\to (Y^*)^\ell$ for positive natural numbers $d$ and $\ell$. To prove such stronger weak Extension Principle, one applies a reduction theorem showing continuous functions between \v{C}ech--Stone remainders essentially depend on one variable. This reduction theorem, initially conjectured in \cite{vD:Prime}, was proved in \cite{Fa:Dimension}.

Let $n\geq 1$, and suppose that $X_1,\ldots,X_n$ and $Y$ are sets. A function $f\colon \prod_{i\leq n} X_i\to Y$ \emph{depends on one variable} on some $Z\subseteq \prod_{i\leq n} X_i$ if there is $i$ and a function $g\colon X_i\to Y$ such that $f(x_1,\ldots,x_n)=g(x_i)$ for all $(x_1,\ldots,x_n)\in Z$. A function is \emph{piecewise elementary} if $\prod_{i\leq n} X_i$ can be written as a finite union of rectangles (i.e., sets of the form $A_1\times\cdot\times A_n$) on which $f$ depends on one variable.

\begin{theorem}\label{thm:dependingone}
All continuous functions from products of compact spaces to \v{C}ech--Stone remainders of locally compact second countable spaces are piecewise elementary. Moreover, the rectangles giving the piecewise elementarity decomposition may be chosen to be clopen.
\end{theorem}
Theorem~\ref{thm:dependingone} is not stated in full generality, and we refer to \cite{Fa:Dimension} for the specifics.

We intend to dualise Theorem~\ref{thm:dependingone}, and thus give the appropriate definition of elementary and piecewise elementary maps. For simplicity, we focus on the case of corona \cstar-algebras and require our blocks to be already clopen. All tensor products are assumed to be minimal tensor products.
\begin{definition}
Let $A,B_1,\ldots,B_n$ be \cstar-algebras, where $A$ is nonunital and separable and each $B_i$ is unital. Let $\Phi\colon\mathcal Q(A)\to\bigotimes B_i$ be a unital $^*$-homomorphism.
\begin{itemize}
 \item Let $p_1,\ldots,p_n$ be projections where $p_i\in B_i$. $\Phi$ is said to be elementary on $(p_1,\ldots,p_n)$ is there is $i\leq n$ and a $^*$-homomorphism $\Psi\colon\mathcal Q(A)\to B_i$ such that for all $a\in\mathcal Q(A)$ we have that
\[
(p_1\otimes \cdots\otimes p_n)\Phi(a)(p_1\otimes \cdots\otimes p_n)=p_1\otimes\cdots p_{i-1}\otimes \Psi(a)\otimes p_{i+1}\otimes\cdots\otimes p_n.
\]
\item $\Phi$ is piecewise elementary if there are natural numbers $k_1,\ldots,k_n$ and projections $p_{i,1},\ldots,p_{i,k_i}\in B_i$ such that $\sum_{j\leq k_i} p_{i,j}=1_{B_i}$ for all $i\leq n$ and with the property that for every tuple $(\ell_1,\ldots,\ell_n)$ with $\ell_i\leq k_i$, $\Phi$ is elementary on $(p_{\ell_1},\ldots,p_{\ell_n})$.
\end{itemize}
\end{definition}
The following statement was already isolated in \cite{Gha:SAW*} in case $n=2$.

\begin{theorem}
Let $A,B_1,\ldots,B_n$ be commutative \cstar-algebras, where $A$ is nonunital and separable and each $B_i$ is unital. If $\Phi\colon\mathcal Q(A)\to\bigotimes B_i$ is a unital $^*$-homomorphism, $\Phi$ is piecewise elementary.
\end{theorem}

To prove that a version of the noncommutative weak Extension Principle holds for maps between powers of coronas, one would need to show that all maps between tensor products of coronas are piecewise elementary. We do not have at the current moment a proof of this statement. 

An even more embarrassing open question related to this line of work is the following, which generalises a question of Simon Wassermann on tensorial primality of the Calkin algebra originally treated in \cite{Gha:SAW*}. 
\begin{question}\label{ques:factorizability}
Let $m<n$ be positive natural numbers. Can there exist separable nonunital \cstar-algebras $A_1,\ldots,A_m$ and $B_1,\ldots,B_n$ such that $\bigotimes_{i\leq m} \mathcal Q(A_i)\cong \bigotimes_{i\leq n} \mathcal Q(B_i)$?
\end{question}

Even though the above is stated for minimal tensor product, the norm one uses to complete algebraic tensor products with should not matter. Moreover, the above question should really be stated for SAW${}^*$-algebras, where `being SAW${}^*$' is a property shared by all coronas of separable nonunital \cstar-algebras. It is the noncommutative analogue of `being a $\beta\N$-space' (see \cite{Gha:SAW*} or \cite[Chapter 15]{Fa:Combinatorial}).

We expect a negative answer to Question~\ref{ques:factorizability}. In the commutative setting, due to results on piecewise elementarity of maps, see \cite{Fa:Dimension} and the notes in Chapter 15 in \cite{Fa:Combinatorial}, we indeed have such an answer. In the noncommutative setting, the question is still open even for the Calkin algebra $\mathcal Q(H)$, where the results of \cite{Gha:SAW*} give a negative answer only in case $m=1$. Studying variants of this question motivated recent work of Farah and Vaccaro (\cite{farah2025probablyisomorphicstructures}) on primality of certain massive von Neumann algebras.

If one focuses on embeddings, the situation is different. The main result of \cite{farah2017calkin} shows that the Calkin algebra $\mathcal Q(H)$ is $\aleph_1$-universal for \cstar-algebras, meaning that all \cstar-algebras of density at most $\aleph_1$ embed into $\mathcal Q(H)$. This corresponds to $\aleph_1$ surjective universality for the compact space $\omega^*$, that is, $\aleph_1$ injective universality (in the category of commutative \cstar-algebras) of $\ell_\infty/c_0$. Differently from the commutative case, universality of $\mathcal Q(H)$ cannot be derived from (model theoretic) saturation. If the Continuum Hypothesis $\CH$ is assumed, one can then embed all tensor products of coronas of separable \cstar-algebras (and much more) into $\mathcal Q(H)$, and therefore into many other coronas (for example, in the corona of the stabilisation of a given unital separable \cstar-algebra). This cannot happen if $\OCA$ and $\MA$ are assumed.
\begin{prp}\label{prop:embeddingstensor}
 Assume $\OCA$ and $\MA$. Let $A$, $B$, $C$ be separable nonunital \cstar-algebras. Let $\otimes_\alpha$ be any \cstar-norm completion of the algebraic tensor product $\mathcal Q(A)\odot\mathcal Q(B)$. Then there is no injective $^*$-homomorphism $\mathcal{Q}(A)\otimes_\alpha \mathcal{Q}(B)\to\mathcal Q(C)$. 
\end{prp}

\begin{proof}
The proof does not rely on the specific norm, but only on the fact that both $\mathcal Q(A)$ and $\mathcal Q(B)$ inject unitally into the algebraic tensor product and that $f\odot g=0$ if and only if $f=0$ or $g=0$, therefore we omit all references to the specific norm and stick to the minimal norm. 

We argue by contradiction, and assume that there is an injective $^*$-homomorphism $\Phi \colon \mathcal{Q}(A)\otimes \mathcal{Q}(B) \to \mathcal{Q}(C)$. We let $(e^A_n)$, $(e^B_n)$ and $(e^C_n)$ be good approximate identities for $A$, $B$, and $C$ respectively.

Let 
\[
\Phi^A=\Phi\restriction \mathcal Q(A)\otimes 1\colon \mathcal Q(A)\to\mathcal Q(C)\text{ and }
\Phi^B=\Phi\restriction 1\otimes Q(B)\colon \mathcal Q(B)\to\mathcal Q(C).
\]
Since both $\Phi^A$ and $\Phi^B$ are injective, applying the $\ncwep$ to $\Phi^A$ and $\Phi^B$ we get that the projection $p$ equals $\Phi(1)$, and so $\Phi^A_p=\Phi^A$ and $\Phi^B_p=\Phi^B$. Let now $\bar I=(I_n)$ be a sparse sequence of intervals as in Notation~\ref{notation2}. Sticking to Notation~\ref{notation2}, we can construct the sets $\mathcal D^A(\bar I)=\prod \mathcal D^A(\bar I)_n$ and $\mathcal D^B(\bar I)=\prod \mathcal D^B(\bar I)_n$ with the following properties: there are intervals $[j_n,k_n]$ and $[j_n',k_n']$ with $\lim j_n=\lim j_n'=\infty$ such that 
\[
\mathcal D^A(\bar I)_n\subseteq (e^C_{k_n}-e^C_{j_n})C(e^C_{k_n}-e^C_{j_n})\text{ and }\mathcal D^B(\bar I)_n\subseteq (e^C_{k'_n}-e^C_{j'_n})C(e^C_{k'_n}-e^C_{j'_n})
\]
and there are functions $\alpha^B_n\colon\mathcal F(\bar I)_n\to\mathcal D^A(\bar I)_n$ and $\alpha^B_n\colon\mathcal F(\bar I)_n\to\mathcal D^B(\bar I)_n$ such that  $\alpha^A=\prod\alpha^A_n \text{ and }\alpha^B=\prod\alpha^B_n$ lift $\Phi^A$ and $\Phi^B$ on $\mathcal F^A(\bar I)$ and $\mathcal F^B(\bar I)$ respectively. Let $(n_\ell)$ and $(m_\ell)$ be increasing sequences of natural numbers such that 
\[
j_{n_\ell}<k_{n_\ell}<j'_{m_\ell}<k'_{m_\ell}<j_{n_\ell+1}<k_{n_\ell+1}<j'_{m_\ell+1}<k'_{m_\ell+1}
\]
for all $\ell\in \N$. Note that $(\prod_\ell \mathcal D^A(\bar I)_{n_\ell})(\prod_\ell \mathcal D^B(\bar I)_{m_\ell})=0$. Pick $f_A\in \mathcal F^A(\bar I)$ supported on $\bigcup_\ell I_{n_\ell}$ and $f_B\in \mathcal F^B(\bar I)$ supported on $\bigcup_\ell I_{m_\ell}$ be two elements such that $1=\norm{\pi_A(f_A)}=\norm{\pi_B(f_B)}$. Then 
\[
\Phi(\pi_A(f_A)\odot \pi_B(f_B))=\Phi^A(\pi_A(f_A))\Phi^B(\pi_B(f_B))=\pi_C(\alpha^A(f_A)\alpha^B(f_B)).
\]
Since $\Phi$ is injective, $\Phi(\pi_A(f_A)\odot \pi_B(f_B))$ has norm $1$, but since $f_A\in \prod_\ell \mathcal D^A(\bar I)_{n_\ell}$ and $f_B\in \prod_\ell \mathcal D^B(\bar I)_{m_\ell}$ we have that $\pi_C(\alpha^A(f_A)\alpha^B(f_B))=0$. This contradiction concludes the proof.
\end{proof}
The following extends Theorem 1.2(1) in \cite{vaccaro2019trivial}.

\begin{corollary}\label{cor:tensors}
Assume $\OCA$ and $\MA$. Let $A$ be a separable nonunital \cstar-algebra. The class of \cstar-algebras embedding into $\mathcal Q(A)$ is not closed under minimal/maximal tensor products.
\end{corollary}
\begin{proof}
$\mathcal Q(A)$ embeds into $\mathcal Q(A)$, yet under $\OCA$ and $\MA$ Proposition~\ref{prop:embeddingstensor} shows that $\mathcal Q(A)\otimes_\gamma\mathcal Q(A)$ cannot embed into $\mathcal Q(A)$ independently on the tensor norm $\gamma$.
\end{proof}

The thesis of Corollary~\ref{cor:tensors} fails under $\CH$, as Parovicenko's theorem (in the commutative setting) or the main result of \cite{farah2017calkin} show that $\ell_\infty/c_0$ and $\mathcal Q(H)$ are injectively universal for the class of abelian \cstar-algebras (resp., all \cstar-algebras) of density $\leq\aleph_1$.
\bibliographystyle{amsplain}
\bibliography{bibliography}

@article{vignati2018rigidity,
	author = {Vignati, A.},
	fjournal = {Annales Scientifiques de l'\'{E}cole Normale Sup\'{e}rieure. Quatri\`eme S\'{e}rie},
	journal = {Ann. Sci. \'{E}c. Norm. Sup\'{e}r. (4)},

	number = {6},
	pages = {1687--1738},
	title = {Rigidity conjectures for continuous quotients},
	volume = {55},
	year = {2022}}

@article{TrivIsoMetric,
	author = {De Bondt, B. and Vignati, A.},
	title = {A metric lifting theorem},
 JOURNAL = {C. R. Math. Acad. Sci. Paris},
  FJOURNAL = {Comptes Rendus Math\'ematique. Acad\'emie des Sciences. Paris},
    VOLUME = {363},
      YEAR = {2025},
     PAGES = {415--424},
}

@article{Fa:AQ,
	author = {Farah, I.},
	doi = {10.1090/memo/0702},
	fjournal = {Memoirs of the American Mathematical Society},
	journal = {Mem. Amer. Math. Soc.},
	number = {702},
	pages = {xvi+177},
	title = {Analytic quotients: theory of liftings for quotients over analytic ideals on the integers},
	volume = {148},
	year = {2000},
}

@article{mckenney2018forcing,
	author = {McKenney, P. and Vignati, A.},
	fjournal = {Journal of Mathematical Logic},
	journal = {J. Math. Log.},
	number = {2},
	pages = {Paper No. 2150006, 73},
	title = {Forcing axioms and coronas of {$\mathrm{C}^*$}-algebras},
	volume = {21},
	year = {2021},
	}

@article {vaccaro2019trivial,
    AUTHOR = {Vaccaro, A.},
     TITLE = {Trivial endomorphisms of the {C}alkin algebra},
   JOURNAL = {Israel J. Math.},
  FJOURNAL = {Israel Journal of Mathematics},
    VOLUME = {247},
      YEAR = {2022},
    NUMBER = {2},
     PAGES = {873--903},
 }

@article{farah2017calkin,
	author = {Farah, I. and Hirshberg, I. and Vignati, A.},
	journal = {Israel J. Math.},
	pages = {287--309},
	publisher = {Springer},
	title = {The {C}alkin algebra is $\aleph_1 $-universal},
	volume = {237},
	year = {2020}}

@article{dow2014non,
	author = {A. Dow},
	journal = {Proc. Amer. Math. Soc.},
	number = {8},
	pages = {2907--2913},
	title = {A non-trivial copy of {$\beta \mathbb N \setminus \mathbb N$}},
	volume = {142},
	year = {2014}}

@article{Ru,
	author = {Rudin, W.},
	journal = {Duke Mathematics Journal},
	pages = {409--419},
	title = {Homogeneity problems in the theory of \v{C}ech compactifications},
	volume = {23},
	year = {1956}}

@article{Fa:Dimension,
	author = {Farah, I.},
	journal = {Top. Appl.},
	pages = {279--297},
	title = {Dimension phenomena associated with {$\beta{\bbN}$}-spaces},
	volume = {125},
	year = {2002}}

@book{vD:Prime,
	author = {van Douwen, E.K.},
	publisher = {Warszawa},
	series = {Dissertationes Mathematicae},
	title = {Prime mappings, number of factors and binary operations},
	volume = {199},
	year = {1981}}

@book{Pede:C*,
	address = {London},
	author = {Pedersen, G.K.},
     TITLE = {{$C^*$}-algebras and their automorphism groups},
    SERIES = {Pure and Applied Mathematics (Amsterdam)},
   EDITION = {Second},
      NOTE = {Edited and with a preface by S\o ren Eilers and Dorte Olesen},
 PUBLISHER = {Academic Press, London},
      YEAR = {2018},
     PAGES = {xviii+520},
}

@article{Arv:Notes,
	author = {Arveson, W.},
	journal = {Duke Math. J.},
	pages = {329--355},
	title = {Notes on extensions of \cstar-algebras},
	volume = 44,
	year = 1977}

@article{BrDoFi,
	author = {Brown, L.G. and Douglas, R.G. and Fillmore, P.A.},
	journal = {Annals of Math.},
	pages = {265--324},
	title = {Extensions of \cstar-algebras and {K}-homology},
	volume = {105},
	year = {1977}}

@article{PhWe:Calkin,
	author = {Phillips, N.C. and Weaver, N.},
	journal = {Duke Math. Journal},
	pages = {185--202},
	title = {The {C}alkin algebra has outer automorphisms},
	volume = {139},
	year = {2007}}

@article {HjeRordam,
    AUTHOR = {Hjelmborg, J. v. B. and R{\o}rdam, M.},
     TITLE = {On stability of {$C^*$}-algebras},
   JOURNAL = {J. Funct. Anal.},
  FJOURNAL = {Journal of Functional Analysis},
    VOLUME = {155},
      YEAR = {1998},
    NUMBER = {1},
     PAGES = {153--170}
}

@article{Ell:Derivations,
	author = {Elliott, G.A.},
	journal = {Ann. of Math. (2)},
	pages = {407--422},
	title = {Derivations of matroid \cstar-algebras. {II}},
	volume = {100},
	year = {1974}}

@book{Black:Operator,
	address = {Berlin},
	author = {Blackadar, B.},
	publisher = {Springer-Verlag},
	series = {Encyclopaedia of Mathematical Sciences},
	title = {Operator algebras},
	volume = {122},
	year = {2006}}

@incollection{Pede:Corona,
	address = {Harlow},
	author = {Pedersen, G.K.},
	booktitle = {Operator {T}heory: {P}roceedings of the 1988 {GPOTS}-{W}abash {C}onference ({I}ndianapolis, {IN}, 1988)},
	pages = {49--92},
	publisher = {Longman Sci. Tech.},
	series = {Pitman Res. Notes Math. Ser.},
	title = {The corona construction},
	volume = {225},
	year = {1990}}

@article{CoFa:Automorphisms,
	author = {Coskey, S. and Farah, I.},
	fjournal = {Transactions of the American Mathematical Society},
	journal = {Trans. Amer. Math. Soc.},
	number = {7},
	pages = {3611--3630},
	title = {Automorphisms of corona algebras, and group cohomology},
	volume = {366},
	year = {2014}}

@article {ArchboldSomerset,
    AUTHOR = {Archbold, R. J. and Somerset, D. W. B.},
     TITLE = {Ideals in the multiplier and corona algebras of a
              {$C_0(X)$}-algebra},
   JOURNAL = {J. Lond. Math. Soc. (2)},
  FJOURNAL = {Journal of the London Mathematical Society. Second Series},
    VOLUME = {85},
      YEAR = {2012},
    NUMBER = {2},
     PAGES = {365--381}
}

@article{Gha:SAW*,
	author = {Ghasemi, S.},
	fjournal = {Glasgow Mathematical Journal},
	journal = {Glasg. Math. J.},
	number = {1},
	pages = {1--5},
	title = {{SAW}* algebras are essentially non-factorizable},
	volume = {57},
	year = {2015}}

@article {Busby,
    AUTHOR = {Busby, R. C.},
     TITLE = {Double centralizers and extensions of {$C\sp{\ast}
              $}-algebras},
   JOURNAL = {Trans. Amer. Math. Soc.},
  FJOURNAL = {Transactions of the American Mathematical Society},
    VOLUME = {132},
      YEAR = {1968},
     PAGES = {79--99}
}

@article {Ng.Minimal,
    AUTHOR = {Kaftal, V. and Ng, P. W. and Zhang, S.},
     TITLE = {The minimal ideal in multiplier algebras},
   JOURNAL = {J. Operator Theory},
  FJOURNAL = {Journal of Operator Theory},
    VOLUME = {79},
      YEAR = {2018},
    NUMBER = {2},
     PAGES = {419--462}
}

@article{FaMcK:Homeomorphisms,
	author = {Farah, I. and McKenney, P.},
	journal = {Proc. Amer. Math. Soc.},
	number = {5},
	pages = {2253--2262},
	title = {Homeomorphisms of \v{C}ech--{S}tone remainders: the zero-dimensional case},
	volume = {146},
	year = {2018}}

@article {KucNgPerera,
    AUTHOR = {Kucerovsky, D. and Ng, P. W. and Perera, F.},
     TITLE = {Purely infinite corona algebras of simple {$C^\ast$}-algebras},
   JOURNAL = {Math. Ann.},
  FJOURNAL = {Mathematische Annalen},
    VOLUME = {346},
      YEAR = {2010},
    NUMBER = {1},
     PAGES = {23--40}
}

@article {Zhang.Riesz,
    AUTHOR = {Zhang, S.},
     TITLE = {A {R}iesz decomposition property and ideal structure of
              multiplier algebras},
   JOURNAL = {J. Operator Theory},
  FJOURNAL = {Journal of Operator Theory},
    VOLUME = {24},
      YEAR = {1990},
    NUMBER = {2},
     PAGES = {209--225}
}

@article{vignati2014algebra,
	author = {Vignati, A.},
	journal = {J. Symb. Log.},
	number = {3},
	pages = {1066--1074},
	title = {An algebra whose subalgebras are characterized by density},
	volume = {80},
	year = {2015}}

@article {OlsenPed.Coronas,
    AUTHOR = {Olsen, C. L. and Pedersen, G. K.},
     TITLE = {Corona {$C^*$}-algebras and their applications to lifting
              problems},
   JOURNAL = {Math. Scand.},
  FJOURNAL = {Mathematica Scandinavica},
    VOLUME = {64},
      YEAR = {1989},
    NUMBER = {1},
     PAGES = {63--86}
}

@unpublished{farah2025coronassa,
      title={Coronas and strongly self-absorbing \cstar-algebras}, 
      author={Farah, I. and Szabó, G.},
      note={arXiv:2411.02274} 
}

@article{Lin2016corona,
	author = {Lin, H. and Ng, P.W.},
	fjournal = {Journal of Functional Analysis},
	journal = {J. Funct. Anal.},
	number = {3},
	pages = {1220--1267},
	title = {The corona algebra of the stabilized {J}iang-{S}u algebra},
	volume = {270},
	year = {2016}}

@article{lin1988ideals,
	author = {Lin, H.},
	journal = {Proc. Amer. Math. Soc.},
	number = {1},
	pages = {239--244},
	title = {Ideals of multiplier algebras of simple {AF} \cstar-algebras},
	volume = {104},
	year = {1988}}

@article {Ng.PICoronasExt,
    AUTHOR = {Ng, P.W.},
     TITLE = {Purely infinite corona algebras and extensions},
   JOURNAL = {J. Noncommut. Geom.},
  FJOURNAL = {Journal of Noncommutative Geometry},
    VOLUME = {16},
      YEAR = {2022},
    NUMBER = {4},
     PAGES = {1363--1395}
}

@article{lin1991simple,
	author = {Lin, H.},
	journal = {Proc. Amer. Math. Soc.},
	number = {3},
	pages = {871--880},
	title = {Simple \cstar-algebras with continuous scales and simple corona algebras},
	volume = {112},
	year = {1991}}

@article {CoronaRigiditySurvey,
	author = {Farah, I. and Ghasemi, S. and Vaccaro, A. and Vignati, A.},
     TITLE = {Corona Rigidity},
   JOURNAL = {Bull. Symb. Log.},
  FJOURNAL = {The Bulletin of Symbolic Logic},
    VOLUME = {31},
      YEAR = {2025},
    NUMBER = {2},
     PAGES = {195--287}
}

@book {Fa:Combinatorial,
    AUTHOR = {Farah, I.},
     TITLE = {Combinatorial set theory of {C}*-algebras},
    SERIES = {Springer Monographs in Mathematics},
 PUBLISHER = {Springer, Cham},
      YEAR = {2019},
     PAGES = {xxx+517},
}

@unpublished{farah2025probablyisomorphicstructures,
      title={Probably isomorphic structures}, 
      author={Farah, I. and Vaccaro, A.},
     note={arXiv:2507.01518}}

@article {KaftalNgZhang.Minimal,
    AUTHOR = {Kaftal, V. and Ng, P. W. and Zhang, S.},
     TITLE = {The minimal ideal in multiplier algebras},
   JOURNAL = {J. Operator Theory},
  FJOURNAL = {Journal of Operator Theory},
    VOLUME = {79},
      YEAR = {2018},
    NUMBER = {2},
     PAGES = {419--462}
}

@article{lin2004simple,
	author = {Lin, H.},
	journal = {Proc. Amer. Math. Soc.},
	number = {11},
	pages = {3215--3224},
	title = {Simple corona \cstar-algebras},
	volume = {132},
	year = {2004}}

@article{VY:WEP, 
title={THE WEAK {E}XTENSION {P}RINCIPLE}, 
DOI={10.1017/jsl.2025.10124}, 
journal={The Journal of Symbolic Logic}, 
author={A. Vignati and D. Yilmaz}, 
year={2025}, 
pages={1–11}}
\end{document}